\documentclass[]{article}

\usepackage{graphicx}
\usepackage{amsmath}
\usepackage{amssymb}
\usepackage{amsthm}
\usepackage{pxfonts}
\usepackage{enumerate}
\usepackage{color}
\usepackage{mathdots}
\usepackage{sectsty}
\usepackage[hidelinks]{hyperref}

\sectionfont{\scshape\centering\fontsize{11}{14}\selectfont}
\subsectionfont{\scshape\fontsize{11}{14}\selectfont}
\usepackage{fancyhdr}

\newcommand\shorttitle{HUA-PICKRELL DIFFUSIONS AND FELLER PROCESSES ON THE BOUNDARY OF THE GRAPH OF SPECTRA}
\newcommand\authors{T. Assiotis}

\newtheorem{thm}{Theorem}[section]
\newtheorem{cor}[thm]{Corollary}
\newtheorem{lem}[thm]{Lemma}
\newtheorem{defn}[thm]{Definition}
\newtheorem{rmk}[thm]{Remark}
\newtheorem{prop}[thm]{Proposition}

\title{\large \bf HUA-PICKRELL DIFFUSIONS AND FELLER PROCESSES ON THE BOUNDARY OF THE GRAPH OF SPECTRA}
\author{\small THEODOROS ASSIOTIS}
\date{}

\begin{document}
\maketitle

\begin{abstract}
We consider consistent diffusion dynamics, leaving the celebrated Hua-Pickrell measures, depending on a complex parameter $s$, invariant. These, give rise to Feller-Markov processes on the infinite dimensional boundary $\Omega$ of the "graph of spectra", the continuum analogue of the Gelfand-Tsetlin graph, via the method of intertwiners of Borodin and Olshanski. In the particular case of $s=0$, this stochastic process is closely related to the $\mathsf{Sine_2}$ point process on $\mathbb{R}$ that describes the spectrum in the bulk of large random matrices. Equivalently, these coherent dynamics are associated to interlacing diffusions in Gelfand-Tsetlin patterns having certain \textit{Gibbs} invariant measures. Moreover, under an application of the Cayley transform when $s=0$ we obtain processes on the circle leaving invariant the \textit{multilevel} Circular Unitary Ensemble. We finally prove that the Feller processes on $\Omega$ corresponding to Dyson's Brownian motion and its stationary analogue are given by explicit and very simple deterministic dynamical systems.
\end{abstract}

\section{Introduction}

The main result of this paper is the construction of a Feller-Markov process on the infinite dimensional boundary $\Omega$ of the \textit{"graph of spectra"}, the continuum analogue of the classical Gelfand-Tsetlin graph, leaving the \textit{Hua-Pickrell measures} on $\Omega$ invariant, by the so called \textit{method of intertwiners}.

 This approach, of constructing such Feller processes, was first introduced by Borodin and Olshanski in \cite{BorodinOlshanski} in order to obtain stochastic dynamics on the boundary of the Gelfand-Tsetlin graph, which describes the branching of irreducible representations of the chain of unitary groups $\mathbb{U}(1)\subset \mathbb{U}(2)\subset \cdots$, that leave the $zw$-measures invariant; these distinguished measures first arose in the problem of the harmonic analysis on the infinite dimensional unitary group $\mathbb{U}(\infty)$, see in particular \cite{OlshanskiHarmonic} for more details. 
 
 The formalism of the intertwiners was then subsequently successfully applied in the case of the infinite symmetric group $S(\infty)$ in \cite{BorodinOlshanskiThoma} where in fact a more complete study of the properties of the resulting infinite dimensional process is possible (in particular its space-time correlation kernels can be computed explicitly) and also very recently by Cuenca in \cite{Cuenca} for the $BC$-type branching graph, which is related to the infinite symplectic $Sp(\infty)$ and orthogonal $O(\infty)$ groups.
 
 However, until now all these applications have been in the discrete setting and this contribution is the first one that deals directly with the continuum. Moreover, it should be noted that in the random matrix setting this is the first time an infinite dimensional Markov process is constructed starting from an arbitrary initial configuration and having the Feller property. Even in the simpler model of Dyson Brownian motion, in the works of Osada (see for example \cite{OsadaTanemura} and references therein) only equilibrium dynamics are considered and also in the tour de force work of Tsai \cite{Tsai}  the initial configuration needs to satisfy a certain balanced assumption. As will become clear, the reason we can achieve this  construction is because we take advantage of all integrable structures underlying this problem. Finally, as the Gelfand-Tsetlin graph degenerates to the graph of spectra under a limiting transition, we expect the dynamics constructed in this paper to be naturally related through a scaling limit (after possibly scaling the parameters as well) with the dynamics considered in \cite{BorodinOlshanski}, although the exact connection remains mysterious for now, see Section \ref{sectionmultilevel}.
 
 We now proceed to give a more detailed, although still informal, exposition of our results. All notions introduced below will be made precise in the relevant sections later on.
 
 We begin in Section \ref{sectionergodicboundary} by recalling several facts about unitarily invariant measures on the space of infinite Hermitian matrices $H$; these are precisely the measures invariant under the action by conjugation of $\mathbb{U}(\infty)$. As in all the settings mentioned above, these measures have a representation theoretic meaning as well, the ergodic invariant measures are in one to one correspondence with (equivalence classes of) spherical representations $\left(T,\xi\right)$ of the infinite dimensional Cartan motion group $G(\infty)=\lim_{N \to \infty}G(N)$ where $G(N)=\mathbb{U}(N) \ltimes H(N)$, the reader is referred to \cite{Pickrell} and \cite{OlshanskiVershik} for more details. The fundamental and indeed very remarkable result in the area, first appearing in Vershik's note \cite{VershikErgodic} where he introduced the so called \textit{ergodic method}, later also proved by Pickrell \cite{Pickrell} and a more detailed exposition of the original proof of Vershik appearing in \cite{OlshanskiVershik}, is the fact that the extremal or ergodic $\mathbb{U}(\infty)$ invariant measures can be characterized explicitly and are parametrized by the infinite dimensional space $\Omega$ defined in (\ref{boundarydefinition}).
 
We then define the "graph of spectra", which is not really a graph in the rigorous sense (that explains our use of quotation marks), but rather a projective chain. This is given by the sequence $\{W^N\}_{N\ge 1}$ of Weyl chambers in $\mathbb{R}^N$ namely $(x_1,\cdots,x_N)\in W^N$ if $x_1\le \cdots \le x_N$ and Markov kernels or links $\Lambda_N^{N+1}:W^{N+1}\to W^N$ given by ratios of Vandermonde determinants $\Delta_N(x)=\prod_{1 \le i <j \le N}^{}(x_j-x_i)$ as follows (some slight care is needed when some of $x$ coordinates coincide see Section 2.2),
\begin{align*}
\Lambda^{N+1}_{N}(x,dy)=\frac{N! \Delta_N(y)}{\Delta_{N+1}(x)}\boldmath{1}(y\prec x)dy,
\end{align*}
where for $y \in W^N,x \in W^{N+1}$ $y\prec x$ denotes interlacing: $x_1\le y_1 \le x_2 \le \cdots \le x_{N+1}$. It turns out that the Feller boundary in the sense of Borodin and Olshanski of this chain can be identified with the space $\Omega$. More precisely (but note that this is not the exact definition of a Feller boundary, some extra conditions are needed), the extreme set of the convex set consisting of sequences of coherent  \textit{probability} measures $\{\mu_N\}_{N\ge 1}$ on $\{W^N\}_{N\ge1}$ namely so that,
\begin{align*}
\mu_{N+1}\Lambda_N^{N+1}=\mu_N \ , \forall N\ge 1,
\end{align*}
can be parametrized by $\Omega$; moreover the Markov kernels $\Lambda_N^{\infty}:\Omega \to W^N \ \forall N\ge 1$  (under certain regularity assumptions) are given explicitly in terms of a single \textit{totally positive function}. We then close this section, following \cite{BorodinOlshanski} with a brief introduction to the main results of the method of intertwiners that we will use later on.

In Section \ref{sectionhuapickrellmeasures}, we introduce the Hua-Pickrell measures $\mu_{HP}^{s,N}$ on $W^N$, where $s$ is a complex parameter, that will be our main focus in this work. These measures were first studied by Hua Luogeng in the $50$'s in his book \cite{Hua} and later in the $80$'s rediscovered by Pickrell \cite{Pickrellmeasure} in the context of Grassmann manifolds, see also Neretin's generalization \cite{Neretin}. Borodin and Olshanski investigated in particular their determinantal properties \cite{BorodinOlshanskiErgodic} and very recently Bufetov and Qiu, see for example \cite{BufetovQiu} and \cite{Qiu}, studied the infinite case (when they can no longer be normalized to be probability measures) and also settled several open problems from \cite{BorodinOlshanskiErgodic}. We will collect several of their properties and key facts, the most fundamental being that they are consistent with the links $\Lambda_N^{N+1}$ above,
\begin{align*}
\mu_{HP}^{s,N+1}\Lambda_N^{N+1}=\mu_{HP}^{s,N} \ , \forall N\ge 1,
\end{align*}
 so that in particular we obtain, a non-extremal or equivalently not a delta function, measure $\mu_{HP}^s$ on $\Omega$. We mention in passing that, we will also give an independent proof of the consistency relation above, that avoids any difficult explicit computations of integrals, using the dynamical approach advocated in this paper.
 
 In Section \ref{sectiondiffusions} we introduce our stochastic dynamics. Akin to the classical case of Dyson's Brownian motion for $\beta=2$ these are given equivalently as a Doob's $h$-transform of one dimensional diffusions (with transition densities in $\mathbb{R}$ denoted by $p_t^{(N),s}$) killed when they intersect i.e. with transition density in $\mathring{W}^N$ given by,
\begin{align*}
e^{-\lambda_{N,s}t}\frac{\Delta_N(y)}{\Delta_N(x)}\det\left(p^{(N),s}_t(x_i,y_j)\right)^N_{i,j=1}dy,
\end{align*}
or as the unique strong solution to the system of Stochastic Differential Equations ($SDEs$) with long range repulsion where the $\{W_i\}_{i=1}^N$ are independent standard Brownian motions,

\begin{align*}
dX_i(t)=\sqrt{2(X_i^2(t)+1)}dW_i(t)+\left[\left(2-2N-2\Re(s)\right)X_i(t)+2\Im(s)+\sum_{j\ne i}^{}\frac{2(X^2_i(t)+1)}{X_i(t)-X_j(t)}\right]dt.
\end{align*}

We prove well-posedness and the Feller property for these processes and most importantly, that for $\Re(s)>-\frac{1}{2}$ the measures $\mu_{HP}^{s,N}$ are their unique invariant probability measures. Namely, if we denote by $P_{HP}^{s,N}(t)$ the Feller semigroups associated to the processes above we have for $\Re(s)>-\frac{1}{2}$,
\begin{align*}
\mu_{HP}^{s,N}P_{HP}^{s,N}(t)=\mu_{HP}^{s,N} \ \ t \ge 0, \forall N \ge 1.
\end{align*}

We then arrive at Section \ref{sectionintertwinings} where, after recalling some necessary results from \cite{InterlacingDiffusions} (we shall give a self-contained proof of these in Section \ref{appendix}) where intertwining relations between determinantal semigroups were studied, we prove our main result, the following consistency relation between the semigroups,
\begin{align*}
P_{HP}^{s,N+1}(t)\Lambda_N^{N+1}=\Lambda_N^{N+1}P_{HP}^{s,N}(t) \ ,\ t \ge 0, \forall N\ge 1.
\end{align*}
We thus, via the formalism of the method of intertwiners, obtain a Feller-Markov process with semigroup $P_{HP}^{s,\infty}(t)$ on $\Omega$ consistent with the stochastic processes on level $N$,
\begin{align*}
P_{HP}^{s,\infty}(t)\Lambda_N^{\infty}=\Lambda_N^{\infty}P_{HP}^{s,N}(t) \ ,\ t \ge 0, \forall N\ge 1 ,
\end{align*}
that has $\mu_{HP}^s$ for $\Re(s)>-\frac{1}{2}$ as its unique invariant probability measure. Since the description of these processes might seem a bit abstract and out of reach, we then discuss a hands on approximation procedure for boundary Feller processes from their finite $N$ analogues. Furthermore, as is by now relatively well known there are other (except the Hua-Pickrell introduced here) multidimensional diffusions consistent with the links $\Lambda_N^{N+1}$. The two most classical and simplest examples being Dyson's Brownian motion and its stationary Ornstein-Uhlenbeck counterpart (see for example \cite{Warren} and for general $\beta$ \cite{RamananShkolnikov}). By the intertwiners formalism, one again obtains a Feller process for each on $\Omega$. It turns out however that, these processes are simple deterministic dynamical systems and we showcase this by the rather down to earth approximation procedure mentioned above, see Subsection \ref{remarkDyson} for more details.

Moving on to Section \ref{sectionmultilevel}, we make the connection to interacting particle systems in $(2+1)$-dimensions. The motivation behind this section is to provide a relation with the discrete dynamics considered by Borodin and Olshanski on the path space of the Gelfand-Tsetlin graph. More precisely, making use of the general results of \cite{InterlacingDiffusions}, we construct consistent dynamics on the \textit{path space} of the graph of spectra leaving the multilevel Hua-Pickrell measures invariant. This path space is given equivalently by infinite interlacing arrays. More specifically, a path of length $N$ is given by a continuous Gelfand-Tsetlin pattern of depth $N$, denoted by $\mathbb{GT}_c(N)$. The diffusion processes $\mathbb{X}^{(N)}$ we construct in $\mathbb{GT}_c(N)$ (note that there must be some interaction between the components in order for the interlacing to remain) are such that if they are started according to a \textit{Gibbs} or \textit{Central} measure (see display (\ref{Gibbs}) for a precise definition) then the projection $\pi_n\mathbb{X}^{(N)}=\left(X_1^{(n)},\cdots,X^{(n)}_n\right)$ on the $n^{th}$ level evolves according to $P_{HP}^{s,n}(t)$. Moreover, in Subsection \ref{sectionCUE} we study how our results transfer to the circle $\mathbb{T}$ under an application of the $Cayley$ transform, which in more generality maps Hermitian matrices to unitary matrices. For the particular case $s=0$, we obtain an interlacing process that leaves the multilevel Circular Unitary Ensemble (CUE) invariant.

In Section \ref{sectionmatrixprocess} we introduce a matrix valued (more precisely Hermitian valued) process whose eigenvalue evolution is that of the system of $SDEs$ considered above.

Finally, in Section \ref{appendix} we give, for completeness of the paper, a self-contained proof of the intertwining relation from \cite{InterlacingDiffusions} that we make use of in Section \ref*{sectionintertwinings}.

\paragraph{Acknowledgements} I would like to thank Jon Warren for some useful comments on an earlier version. I am also very grateful to Professor A. Borodin for historical remarks and pointers and especially to Professor G. Olshanski for some stimulating comments which led to several improvements. I am also thankful to N. Demni for useful comments. Finally, I would like to thank two anonymous referees for many useful comments and suggestions that led to significant improvements in exposition. Financial support from EPSRC through the MASDOC DTC grant number EP/HO23364/1 is gratefully acknowledged.

\section{Ergodic measures and the boundary of the graph of spectra}\label{sectionergodicboundary}

\subsection{Ergodic unitarily invariant measures}

We begin by recalling some useful facts about unitarily invariant measures on the space of infinite Hermitian matrices. We mainly follow \cite{BorodinOlshanskiErgodic} and \cite{OlshanskiVershik},the connection to the graph of spectra will be clarified in the sequel. So, let $\mathbb{U}(N)$ be the group of $N\times N$ unitary matrices. Let $H(N)$ denote the space of $N \times N$ Hermitian matrices. Define the Cayley transform that maps $X \in H(N)$ to $U\in \mathbb{U}(N)$ by,
\begin{align*}
X\mapsto U=\frac{i-X}{i+X}.
\end{align*}
We denote this bijective map by $\mathfrak{C}$ and by $\pi^{N+1}_N$ the "cutting corner" projection from $H(N+1)$ to $H(N)$: $\pi^{N+1}_N\left[\left(h_{ij}\right)^{N+1}_{i,j=1}\right]=\left(h_{ij}\right)^{N}_{i,j=1}$. Finally we will write $\mathsf{eval}_N:H(N) \to W^N$ for the map on Hermitian matrices $H(N)$ defined by $\mathsf{eval}_N(H)=(x_1\le \cdots \le x_N)$ where the $(x_1\le \cdots \le x_N)$ are the ordered eigenvalues of the matrix $H$.

 Moving on, we let $H$ denote the projective limit $\underset{\leftarrow}{\lim}H(n)$, the space of all infinite Hermitian matrices which can be naturally identified as a topological vector space with the infinite product $\mathbb{R}^{\infty}=\mathbb{R}\times \mathbb{R}\times \mathbb{R}\times\cdots$ by:
 \begin{align*}
 H\ni X \mapsto \{ X_{ii}\}\sqcup \{\Re{X_{ij}}, \Im{X_{ij}} \}.
 \end{align*}
Moreover, let $H(\infty)$ denote the inductive limit, $\lim_{N \to \infty} H(N)$, the space of $\infty \times \infty$ Hermitian matrices with finitely many non-zero entries and similarly we consider $\mathbb{U}(\infty)=\lim_{N \to \infty} \mathbb{U}(N)$ the inductive limit unitary group. With these definitions in place, there exists a pairing,
\begin{align*}
H(\infty) \times H \to \mathbb{R} \ ,\ \ (A,X) \mapsto Tr(AX).
\end{align*}
Now, for a Borel probability measure $M$ on $H$ define its Fourier transform as the function on $H(\infty)$ denoted by,
\begin{align*}
\hat{M}(A)=F_{M}(A)= \int_{H}e^{iTr(AX)}M(dX) \ \textnormal{ for } \ A\in H(\infty).
\end{align*}
The group $\mathbb{U}(\infty)$ acts on both $H(\infty)$ and $H$ by conjugation and the pairing of the two spaces is $\mathbb{U}(\infty)$ invariant. Observe that a matrix in $H(\infty)$ can be brought by conjugation to a diagonal matrix $\textnormal{diag}(r_1,r_2,\cdots)$ with finitely many non-zero entries. Thus, the Fourier transform of $\mathbb{U}(\infty)$ invariant probability measures on $H$, that we denote by $\mathcal{M}_p^{\mathbb{U}(\infty)-inv}\left(H\right)$, is uniquely determined by its values on the diagonal matrices from $H(\infty)$. It is a remarkable fact that, extremal or ergodic $\mathbb{U}(\infty)$ (these notions are of course equivalent see for example Proposition 1.3 of \cite{OlshanskiVershik}) invariant probability measures, $\mathcal{M}_p^{\mathbb{U}(\infty)-erg}\left(H\right)$, can be explicitly characterized. Define the space $\Omega$ by,
\begin{align}\label{boundarydefinition}
\Omega&=\bigg\{\omega=(\alpha^+,\alpha^-,\gamma_1,\delta)\in \mathbb{R}^{2\infty+2}=\mathbb{R}^\infty \times \mathbb{R}^\infty \times \mathbb{R}  \times \mathbb{R}|\nonumber\\
\alpha^+&=(\alpha_1^+\ge \alpha_2^+\ge \cdots \ge 0) \ \ \alpha^-=(\alpha_1^-\ge \alpha_2^-\ge \cdots \ge 0)\nonumber\\
\gamma_1 &\in \mathbb{R} \ \ \sum_{}^{}(\alpha_i^+)^2 + \sum_{}^{}(\alpha_i^-)^2\le \delta \bigg\}
\end{align}
and moreover let $\gamma_2=\delta-\sum_{}^{}(\alpha_i^+)^2 - \sum_{}^{}(\alpha_i^-)^2$. We note that $\Omega$ is a locally compact metrizable topological space with a countable base. Finally, write $F_{\omega}$ for,
\begin{align*}
F_{\omega}(x)=e^{i\gamma_1x-\frac{\gamma_2}{2}x^ 2}\prod_{k=1}^{\infty}\frac{e^{-i\alpha_k^+x}}{1-i\alpha_k^+x}\prod_{k=1}^{\infty}\frac{e^{i\alpha_k^-x}}{1+i\alpha_k^-x} \ .
\end{align*}
Observe that we have the estimate:
\begin{align*}
\frac{e^{i a x}}{1-iax}=1-\frac{3}{2}a^2x^2+O\left(a^3\right) \ \textnormal{ as } a \to 0.
\end{align*}
Thus, since $ \sum_{}^{}(\alpha_i^+)^2 + \sum_{}^{}(\alpha_i^-)^2<\infty$ the function $F_{\omega}(x)$ converges for all $x \in \mathbb{R}$ for any $\omega \in \Omega$; with the result being a continuous function. Moreover, observe that for any fixed $x\in \mathbb{R}$,  $F_{\omega}(x)$ as a function of $\omega\in \Omega$ is continuous.

The following fundamental theorem was first stated and a proof was outlined by Vershik in \cite{VershikErgodic}. It was later also proven by Pickrell \cite{Pickrell} by exploiting the connection to total positivity. A more detailed exposition of the original method of \cite{VershikErgodic} was subsequently given by Olshanski and Vershik in \cite{OlshanskiVershik}, (see also Defosseux \cite{Defosseux}).
\begin{thm}\label{ergodic}
There exists a parametrization of ergodic/extremal $\mathbb{U}(\infty)$-invariant probability measures on the space $H$, $\mathcal{M}_p^{\mathbb{U}(\infty)-erg}\left(H\right)$,  by the points of the space $\Omega$. Given $\omega$ the characteristic function of the ergodic measure $M_{\omega}$ is given by,
\begin{align*}
\int_{X \in H}^{}e^{iTr(\textnormal{diag}(r_1,\cdots,r_n,0,0,\cdots)X)}M_{\omega}(dX)=\prod_{j=1}^{n}F_{\omega}(r_j).
\end{align*}
\end{thm}

\begin{rmk}\label{DescriptionofErgodicmeasuresremark}
We observe that the characteristic function $F_{\omega}$ of an ergodic measure $M_{\omega}$ is given as a product of characteristic functions of simpler measures, with only one non-zero parameter, that we call elementary. Equivalently any ergodic measure is given as a convolution of elementary ergodic ones. More precisely writing this in terms of a sum of independent random Hermitian matrices:
\begin{align*}
\gamma_1 \textnormal{Id}+G^{\gamma_2}+\sum_{k\ge 1}^{}\alpha_k^+\left[-\textnormal{Id}+\zeta^*(k)\zeta(k)\right]+\sum_{k\ge 1}^{}\left(-\alpha_k^-\right)\left[-\textnormal{Id}+\xi^*(k)\xi(k)\right].
\end{align*}
Here, $G^{\gamma_2}$ is an infinite GUE matrix, namely the entries $G_{ii}^{\gamma_2}$ and $\Re G_{ij}^{\gamma_2}$, $\Im G_{ij}^{\gamma_2}$ are independent normal random variables of mean 0 and variance $\gamma_2$ subject to the Hermitian constraint. Moreover, the $\zeta(k)$ and $\xi(k)$ are independent infinite row vectors whose entries are i.i.d. complex normal random variables. For more details see Remarks 2.10-2.13 of \cite{OlshanskiVershik} and also Defosseux \cite{Defosseux} Theorem 2.7.
\end{rmk}

The following two notions will be useful in what follows.
\begin{defn}
A real non-negative measurable function $\phi(x)$ on $\mathbb{R}$ such that $\int_{\mathbb{R}}\phi(x)dx=1$ is called \textit{totally positive} if for $n\ge 1$ and $x_1<\cdots<x_n$ and $y_1<\cdots<y_n$,
\begin{align*}
\det\left(\phi(x_{i}-y_j)\right)^n_{i,j=1}\ge 0.
\end{align*}
\end{defn}
\begin{defn}
A \textit{real smooth non-negative} function $\phi(x)$ on $\mathbb{R}$ such that $\int_{\mathbb{R}}\phi(x)dx=1$ is called \textit{extended totally positive} if,
\begin{align}\label{extendedtotallypositiveinequalities}
\det\left(\phi^{(i-1)}(x_{n+1-j})\right)^n_{i,j=1}\ge 0 \ ,\ n=1,2, \cdots \textnormal{ and } x_1 < \cdots < x_n.
\end{align}
\end{defn}

It can be easily shown, see Proposition 7.6 part (i) of \cite{OlshanskiVershik}, that a smooth totally positive function is actually extended totally positive. On the other hand, as the terminology suggests, an extended totally positive function is also totally positive, see Proposition IV.2.3 of \cite{Faraut}, also Proposition 7.6 part (ii) of \cite{OlshanskiVershik}. The elegant argument for this goes as follows: if we convolve an extended totally positive function $\phi$ with a Gaussian of variance $\mathfrak{s}^2>0$, then the inequalities in (\ref{extendedtotallypositiveinequalities}) become strict, see for example Proposition 7.6 part (ii) of \cite{OlshanskiVershik}. Then, by Theorem 2.1 on page 50 of \cite{Karlin} the convolved function is in fact totally positive. We conclude by sending $\mathfrak{s}^2\to 0$ to recover $\phi$.

Now, by Theorem 7.7 of \cite{OlshanskiVershik} (see also Proposition 7.6 part (ii) therein) for $\omega \in \Omega$ with $\gamma_2(\omega)>0$, the function $\phi=\phi_{\omega}$ such that its Fourier transform is given by,
\begin{align*}
\hat{\phi}(\xi)=\hat{\phi}_{\omega}(\xi)=F_{\omega}(\xi),
\end{align*}
is extended totally positive and thus also totally positive.

\subsection{The graph of Spectra and its Boundary}
We start by setting up some notation. Write $x=(x_1,\cdots,x_N)\in W^N$ if $x_1\le \cdots \le x_N$ and furthermore write $W^{N,N+1}(x)$ for the set of $y\in W^N$ that interlace with $x\in W^{N+1}$ i.e. $x_1 \le y_1 \le x_2 \le \cdots \le y_N \le x_{N+1}$ (we will also denote this suppressing any dependence on $N$ by $y\prec x$). We define the Markov kernel for $x \in \mathring{W}^{N+1}$, the interior of $W^{N+1}$,
\begin{align*}
\Lambda^{N+1}_{N}(x,dy)=\frac{N! \Delta_N(y)}{\Delta_{N+1}(x)}\boldmath{1}(y\in W^{N,N+1}(x))dy,
\end{align*}
where $\Delta_{N}(y)=\prod_{1 \le i <j \le N}^{}(y_j-y_i)$. In fact, the Markov kernel above has an interpretation as a conditional distribution for matrix eigenvalues, the first published proof of this fact was given by Baryshnikov (see Proposition 4.2 in \cite{Baryshnikov}) in the random matrix literature (see also Proposition 3.1 in \cite{OlshanskiProjections} and the historical comments therein). Namely, it is the law of:
\begin{align}\label{VandermondeKernelMatrixRep}
\mathsf{eval}_N\left(\pi_N^{N+1}\left[U^*\textnormal{diag}\left(x_1,\cdots,x_{N+1}\right)U\right]\right)
\end{align}
where $U$ is a Haar distributed unitary matrix from $\mathbb{U}(N+1)$.
Observe that the expression (\ref{VandermondeKernelMatrixRep}) makes sense for arbitrary $x \in W^{N+1}$. Thus, for any $x \in W^{N+1}$ we take as the definition of $\Lambda_N^{N+1}(x,\cdot)$ the law of (\ref{VandermondeKernelMatrixRep}).

We will see in the proof of the lemma below that this definition coincides with the weak limit of $\Lambda_N^{N+1}\left(x^{(n)},\cdot\right)$ for $\{x^{(n)}\}_{n}\in \mathring{W}^{N+1}$ converging to $x$. Denote by $C_0\left(W^N\right)$ the space of continuous functions on $W^N$ vanishing at infinity. 
\begin{lem}\label{Feller1}
$\Lambda_N^{N+1}$ is a \textit{Feller} kernel i.e.,
\begin{align*}
\Lambda_N^{N+1}f \in C_0\left(W^{N+1}\right), \forall f \in C_0\left(W^{N}\right).
\end{align*}
\end{lem}
\begin{proof}

We have:
\begin{align*}
\left[\Lambda_N^{N+1}f\right](x_1,\cdots,x_{N+1})=\mathbb{E}_{\mathbb{U}(N+1)}\left[f\left[\mathsf{eval}_N\left(\pi_N^{N+1}\left[U^*\textnormal{diag}\left(x_1,\cdots,x_{N+1}\right)U\right]\right)\right]\right].
\end{align*}
Thus, if we take any sequence $x^{(n)} \in {W}^{N+1}$ converging to some $x \in W^{N+1}$ by the dominated convergence theorem and continuity of all functions involved in the representation above we obtain:
\begin{align*}
\left[\Lambda_N^{N+1}f\right](x^{(n)}_1,\cdots,x^{(n)}_{N+1}) \to \left[\Lambda_N^{N+1}f\right](x_1,\cdots,x_{N+1}).
\end{align*}
In particular, we have the weak convergence of probability measures:
\begin{align*}
\Lambda^{N+1}_{N}(x^{(n)},\cdot)\overset{}{\rightharpoonup
}\Lambda^{N+1}_{N}(x,\cdot).
\end{align*}
Now, we show that as $x^{(n)} \to \infty$ we have $\left[\Lambda_N^{N+1}f\right](x^{(n)}_1,\cdots,x^{(n)}_{N+1}) \to 0$. Without loss of generality assume $x^{(n)}_{N+1} \to \infty$. If $x_N^{(n)}\to \infty$ as well, necessarily by interlacing of eigenvalues we have:
\begin{align*}
\mathsf{eval}_N\left(\pi_N^{N+1}\left[U^*\textnormal{diag}\left(x_1,\cdots,x_{N+1}\right)U\right]\right)\to \infty.
\end{align*}
Then the result follows immediately by the fact that $f \in C_0\left(W^{N}\right)$ and the dominated convergence theorem.
Now, assume $x_N^{(n)}$ remains bounded. We first take for each $n$ a sequence $\{x^{(n),m}\}_{m}\in \mathring{W}^{N+1}$ such that $\underset{m\to \infty}{\lim}x^{(n),m}=x^{(n)}$. For $z \in \mathring{W}^{N+1}$ we have using the explicit expression:
\begin{align}
\left[\Lambda_N^{N+1}f\right](z_1,\cdots,z_{N+1})=\frac{N!\int_{z_1}^{z_2}\cdots \int_{z_{N}}^{z_{N+1}}\Delta_N(y)f(y)dy_1\cdots dy_N}{\Delta_{N+1}(z)}.
\end{align}
Applying the mean-value theorem, a total of $N$ times, successively in the variables $z_{N+1},z_{N},\cdots, z_2$ to the function:
\begin{align*}
F_{N+1}(z_1,\cdots,z_{N+1})=\int_{z_1}^{z_2}\int_{z_2}^{z_3}\cdots \int_{z_N}^{z_{N+1}}\Delta_N(y)f(y)dy_1\cdots dy_N
\end{align*}
we obtain:
\begin{align}
\left[\Lambda_N^{N+1}f\right](z_1,\cdots,z_{N+1})=\frac{N!\prod_{i=1}^{N}(z_{i+1}-z_i)\Delta_N(\xi)f(\xi)}{\Delta_{N+1}(z)}
\end{align}
for some $(\xi_1, \cdots, \xi_N)$ such that $z_1<\xi_1<z_2<\cdots<\xi_N<z_{N+1}$. Moreover, note that the interlacing constraints for  $i=1,\cdots,N-1$ and $l=1,\cdots, N-i$ imply:
\begin{align}\label{MVTconstraintsVandermonde}
\frac{|\xi_{i+l}-\xi_i|}{|z_{i+l+1}-z_i|}\le 1.
\end{align}
Then, we have:
\begin{align*}
\left[\Lambda_N^{N+1}f\right](x^{(n)}_1,\cdots,x^{(n)}_{N+1})&=\underset{m\to \infty}{\lim}\left[\Lambda_N^{N+1}f\right](x^{(n),m}_1,\cdots,x^{(n),m}_{N+1})\\
&=\underset{m\to \infty}{\lim}\frac{N!\prod_{i=1}^{N}(x^{(n),m}_{i+1}-x^{(n),m}_i)\Delta_N(\xi^{(n),m})f(\xi^{(n),m})}{\Delta_{N+1}(x^{(n),m})}\\
&=\underset{m\to \infty}{\lim}\frac{N!\prod_{i=1}^{N}(x^{(n),m}_{i+1}-x^{(n),m}_i)\Delta_N(\xi^{(n),m})}{\Delta_{N+1}(x^{(n),m})}f\left(\xi^{(n)}\right).
\end{align*}
If $\xi_N^{(n)}\to \infty$, then $f\left(\xi^{(n)}\right)\to 0$ and moreover since we have uniformly in $n$ and $m$ the bound from the constraints (\ref{MVTconstraintsVandermonde}):
\begin{align*}
\left|\frac{\prod_{i=1}^{N}(x^{(n),m}_{i+1}-x^{(n),m}_i)\Delta_N(\xi^{(n),m})}{\Delta_{N+1}(x^{(n),m})}\right|\le 1,
\end{align*}
we get $\left[\Lambda_N^{N+1}f\right](x^{(n)}_1,\cdots,x^{(n)}_{N+1}) \to 0$. Now, suppose $\xi_N^{(n)}$ remains bounded. Note that, we have:
\begin{align*}
\underset{n\to \infty}{\lim}\underset{m\to \infty}{\lim}\frac{1}{\left(x_{N+1}^{(n),m}-x_{N-1}^{(n),m}\right)}=0
\end{align*} 
since $x_{N-1}^{(n)}\le x_{N}^{(n)}$ remains bounded. While on the other hand, since $\xi_N^{(n)}$ is bounded, we have that:
\begin{align*}
\left|\frac{\left(x_{N+1}^{(n),m}-x_{N-1}^{(n),m}\right)\prod_{i=1}^{N}(x^{(n),m}_{i+1}-x^{(n),m}_i)\Delta_N(\xi^{(n),m})f\left(\xi^{(n)}\right)}{\Delta_{N+1}(x^{(n),m})}\right|
\end{align*}
remains bounded from which the result follows.

\end{proof}

We will now consider the projective limit of this system $(W^N,\Lambda_N^{N+1})_{N\ge 1}$ in the measurable category $\mathcal{B}$. 
\begin{defn}
The category $\mathcal{B}$ consists of objects given by standard Borel spaces and morphisms given by Markov kernels that we will also call links. Such a kernel $\Lambda: X \to Y$ between two standard Borel spaces $X$ and $Y$ is a function $\Lambda\left(x,\mathcal{Y}\right)$ where $x$ ranges over $X$ and $\mathcal{Y}$ ranges over measurable subsets of $Y$ such that $\Lambda\left(x,\cdot\right)$ is a probability measure on $Y$ for any fixed $x \in X$ and $\Lambda\left(\cdot,\mathcal{Y}\right)$ is a measurable function on $X$ for each fixed $\mathcal{Y}$.
\end{defn}

\begin{defn}\label{DefinitionBoundary}
The limit object $W^{\infty}$ of $(W^N,\Lambda_N^{N+1})_{N\ge 1}$ in $\mathcal{B}$ is understood in the following sense: It consists of an object $W^{\infty}=\underset{\leftarrow}{\lim}W^N$ and links $\Lambda_N^{\infty}: W^{\infty} \to W^N$ such that $\Lambda_N^{\infty} \Lambda_K^N=\Lambda^{\infty}_K, \ \forall K<N$. Moreover if an object $\tilde{W}^{\infty}$ and links $\tilde{\Lambda}_N^{\infty}:\tilde{W}^{\infty}\to W^N$ satisfy the same condition, then there exists a unique link $\Lambda^{\tilde{W}^{\infty}}_{W^{\infty}}:\tilde{W}^{\infty}\to W^{\infty}$ such that  $\tilde{\Lambda}_N^{\infty}=\Lambda^{\tilde{W}^{\infty}}_{W^{\infty}}\Lambda^{\infty}_{N}$. By a general result of Winkler, see Theorem 4.1.3 in \cite{Winkler}, the limit exists and it is unique up to a Borel isomorphism (more generally this fact holds for arbitrary standard Borel spaces in place of the Weyl chambers $W^N$). We will call $W^{\infty}$ the boundary of the system $(W^N,\Lambda_N^{N+1})_{N\ge 1}$.
\end{defn}

In fact, the boundary coincides with the following construction: Note that the links induce the chain of affine mappings: 
\begin{align*}
\cdots\to \mathcal{M}_p\left(W^{N+1}\right)\to\mathcal{M}_p\left(W^{N}\right)\to \cdots \to \mathcal{M}_p\left(W^{2}\right)\to \mathcal{M}_p\left(W^{1}\right) ,
\end{align*} 
where $\mathcal{M}_p\left(W^N\right)$ is the simplex of probability measures on $W^N$ equipped with the weak topology. Consider the space $\mathcal{W}=\prod_{N=1}^{\infty}\mathcal{M}_p\left(W^N\right)$ with the product topology and define the inverse system of simplices (not to be confused with the limit in the measurable category):
\begin{align*}
\underset{\leftarrow}{\lim}\mathcal{M}_p\left(W^N\right)=\{ (\mu_N)_{N\ge 1} \in \mathcal{W}: \mu_{N+1}\Lambda_N^{N+1}=\mu_N \ , \forall N\},
\end{align*}
consisting of coherent sequences of measures. By Theorem 3.2.3 in \cite{Winkler} (see also step 3 in the proof of Theorem 4.1.3 therein) the convex set $\underset{\leftarrow}{\lim}\mathcal{M}_p\left(W^N\right)$ is actually a Polish simplex. Moreover, by steps 3 and 4 in the proof of Theorem 4.1.3 on page 103 of \cite{Winkler} (see also second paragraph on page 109 of \cite{Winkler}) its extreme points coincide with $W^{\infty}$ (in fact this is how $W^{\infty}$ is constructed):
\begin{align}
W^{\infty}=\underset{\leftarrow}{\lim}W^N=\textnormal{Ex}\left(\underset{\leftarrow}{\lim}\mathcal{M}_p\left(W^N\right)\right).
\end{align}
Thus, the boundary consists of extremal coherent sequences of (probability) measures. Moreover, for $\mathfrak{w}\in W^{\infty}$ such that $\mathfrak{w}=(\mu_N)_{N\ge 1} \in \textnormal{Ex}\left(\underset{\leftarrow}{\lim}\mathcal{M}_p\left(W^N\right)\right)$ the links are given by $\Lambda^{\infty}_N\left(\mathfrak{w},\cdot\right)=\mu_N\left(\cdot\right)$. 
\begin{defn}
In the setting of Definition \ref{DefinitionBoundary}, if moreover all the links $\{\Lambda_N^{N+1}\}_{N\ge 1}$ and $\{\Lambda_N^{\infty}\}_{N\ge 1}$ are Feller, namely map continuous functions vanishing at infinity to continuous functions vanishing at infinity, we will say that $W^{\infty}$ is the Feller boundary of $\{W^N\}_{N\ge 1}$.
\end{defn}
Then, we have the following proposition (proven in this subsection after several preliminaries),

\begin{prop}\label{boundary}
$W^{\infty}=\Omega$ is the Feller boundary of $\{W^N\}_{N\ge 1}$.
\end{prop}

We start by recalling the following crucial observation originally made (in published form) by Borodin and Olshanski in \cite{BorodinOlshanskiErgodic} (see graph of spectra remarks pages 30-31 of \cite{BorodinOlshanskiErgodic}). Let $\mathfrak{M}$ be any $\mathbb{U}(\infty)$ invariant probability measure on $H$ and let $\mu^{\mathfrak{M}}_N=\left(\mathsf{eval}_N\circ \pi_N^{\infty}\right)_*\mathfrak{M}$ be the (ordered) radial part of the projection $(\pi_N^{\infty})_*(\mathfrak{M})$ of $\mathfrak{M}$ on $H(N)$, i.e a measure on $W^N$. Then, $\forall N\ge 1$,
\begin{align*}
\mu^{\mathfrak{M}}_{N+1}\Lambda_N^{N+1}=\mu^{\mathfrak{M}}_N.
\end{align*}
Conversely, any coherent sequence of probability measures $\{\mu_N\}_{N\ge 1}$ comes from a $\mathbb{U}(\infty)$ invariant measure $\mathfrak{M}$. A proof of these statements immediately follows also from Lemma 3.3 and Lemma 3.8 of \cite{Defosseux} for example (see also Proposition 3.1 of \cite{OlshanskiProjections}). Thus, there exists a bijection between coherent measures and $\mathbb{U}(\infty)$ invariant probability measures on $H$. More formally, we have that  $\mathcal{M}_p\left(\Omega\right)=\mathcal{M}_p\left(\mathcal{M}_p^{\mathbb{U}(\infty)-erg}\left(H\right)\right)=\mathcal{M}_p^{\mathbb{U}(\infty)-inv}\left(H\right)$. With this identification consider the map between convex sets $\Phi:\mathcal{M}_p\left(\Omega\right)\to \underset{\leftarrow}{\lim}\mathcal{M}_p(W^N) $:
\begin{align*}
\Phi\left(\mathfrak{M}\right)=\big\{ \left(\mathsf{eval}_N\circ \pi_N^{\infty}\right)_*\mathfrak{M}\  \big\}_{N\ge 1},
\end{align*}
which is an affine bijection and hence we have the following lemma
(the reader obviously notices that all the hard work is transferred from Theorem \ref*{ergodic}, which we are essentially reinterpreting following \cite{BorodinOlshanskiErgodic}),
\begin{lem}\label{bijection}
We have a bijection between $\Omega$ and $\textnormal{Ex}\left(\underset{\leftarrow}{\lim}\mathcal{M}_p\left(W^N\right)\right)$.
\end{lem}

We make this more explicit and we begin by defining the following Markov kernels $\Lambda_N^{\infty}$ from $\Omega$ to $W^N$ for $\omega\in \Omega$ with $\gamma_2(\omega)>0$,
\begin{align}\label{boundarykernel}
\Lambda_N^{\infty}(\omega,dx)=\left(\prod_{k=1}^{N-1}\frac{1}{k!}\right)\det \left(\phi_{\omega}^{(j-1)}(x_{N+1-i})\right)^N_{i,j=1}\Delta_N(x)dx,
\end{align}
from $\Omega$ to $W^N$ where $\phi_{\omega}$ as before is such that $\hat{\phi}_{\omega}(\xi)=F_{\omega}(\xi)$. Obviously, $\Lambda_N^{\infty}(\cdot,dx)$ is measurable on $\Omega$. Moreover, the positivity property, $\Lambda_N^{\infty}(\omega,dx) \ge 0$,  immediately follows from the fact that $\phi_{\omega}$ is extended totally positive. To obtain the following coherency relation
\begin{align*}
\Lambda_{N+1}^{\infty}\Lambda_N^{N+1}=\Lambda_N^{\infty},
\end{align*}
observe that,
\begin{align*}
\Lambda_{N+1}^{\infty}\Lambda_N^{N+1}(\omega,dy)&=\left(\prod_{k=1}^{N-1}\frac{1}{k!}\right)\Delta_{N}(y)\int_{-\infty}^{y_1}\cdots \int_{y_N}^{\infty}\det \left(\phi_{\omega}^{(j-1)}(x_{N+2-i})\right)^{N+1}_{i,j=1}dx_1\cdots dx_{N+1}dy\\
&=\left(\prod_{k=1}^{N-1}\frac{1}{k!}\right)\Delta_{N}(y)\det \left(\phi_{\omega}^{(j-1)}(y_{N+1-i})\right)^N_{i,j=1}dy.
\end{align*}
To see this first note that the integral is equal to:
\begin{align*}
\det \begin{bmatrix}
&\int_{y_N}^{\infty}\phi_{\omega}(x_{N+1})dx_{N+1}&\cdots&-\phi^{(N-1)}_{\omega}(y_N)\\
&\vdots&\ddots&\vdots\\
&\int_{y_1}^{y_2}\phi_{\omega}(x_2)dx_2&\cdots&\phi^{(N-1)}_{\omega}(y_2)-\phi^{(N-1)}_{\omega}(y_1)\\
&\int_{-\infty}^{y_1}\phi_{\omega}(x_1)dx_1&\cdots&\phi^{(N-1)}_{\omega}(y_1)
\end{bmatrix}_{(N+1) \times (N+1)}.
\end{align*}
Now successively add row $i$ to row $i-1$, starting from $i=N+1$. The identity then follows from the fact that the first row has a $1$ as its first entry since $\int_{-\infty}^{\infty}\phi_{\omega}(x)dx=1$ and 0's elsewhere.
Finally, to see that $\Lambda_N^{\infty}$ is correctly normalized, i.e. $\Lambda_N^{\infty}1=1$, observe that from the coherency relation $\Lambda_N^{\infty}\Lambda_1^N=\Lambda^{\infty}_1$ and the facts that $\Lambda_1^N1=1$ and $\Lambda_1^{\infty}1=1$ this follows immediately.

We now extend the definition to arbitrary $\omega \in \Omega$. We first observe, that in fact for $\omega$ with $\gamma_2(\omega)>0$ if we consider the measure $M_{\omega}(dX)$ on $H$ with characteristic function $F_{\omega}$ as in Theorem \ref{ergodic}, then (see proof of Theorem 7.7 of \cite{OlshanskiVershik}) $\Lambda^{\infty}_{N}(\omega,dx)$ is the radial part of its projection on $H(N)$, more formally $\Lambda^{\infty}_{N}(\omega,dx)=\left(\mathsf{eval}_N\circ \pi_N^{\infty}\right)_*M_{\omega}(dx)$. In particular, for any $\omega\in \Omega$ we can define $\Lambda^{\infty}_{N}(\omega,dx)$ as the radial part of the projection of $M_{\omega}$ on $H(N)$ or equivalently as the unique weak limit, this essentially follows from Levy's continuity theorem and will also be detailed in Lemma \ref{FellerBoundary} below, as $\omega_{\gamma_2}(n)\to \omega$ (where $\{\omega_{\gamma_2}(n)\}_n$ is any sequence in $\Omega$ such that $\gamma_2\left(\omega_{\gamma_2}(n)\right)>0$ and $\omega_{\gamma_2}(n)\to \omega$) of $\Lambda^{\infty}_{N}(\omega_{\gamma_2}(n),dx)$ namely,
\begin{align*}
\Lambda^{\infty}_{N}(\omega_{\gamma_2}(n),dx)\overset{}{\rightharpoonup
}\Lambda^{\infty}_{N}(\omega,dx), \ \textnormal{as} \ n \to \infty.
\end{align*}

Hence, we have obtained the following lemma,
\begin{lem}\label{consistency}
For all $\omega\in \Omega, N\ge 1$ the kernels $\Lambda^{\infty}_N(\omega,dx)$ are Markov and satisfy,
\begin{align*}
\Lambda^{\infty}_{N+1}\Lambda^{N+1}_{N}=\Lambda_N^{\infty}.
\end{align*}
\end{lem}

Note that, see Remark \ref{DescriptionofErgodicmeasuresremark}, for $\gamma_1(\omega),\alpha_i^+(\omega),\alpha_i^-(\omega)=0$ then $\Lambda_N^{\infty}(\omega,dx)$ is just the $N$-particle $GUE$ with variance $\gamma_2$. Moreover, for $\gamma_2(\omega),\alpha_i^+(\omega),\alpha_i^-(\omega)=0$ then $\Lambda_N^{\infty}(\omega,dx)$ is the delta measure on the scalar matrix $\gamma_1(\omega) \textnormal{Id}_N$, in particular $\Lambda_N^{\infty}(\omega,dx)$ need not have a smooth density with respect to Lebesgue measure. As already anticipated, these kernels are Feller,

\begin{lem}\label{FellerBoundary}
The kernels $\{\Lambda_N^{\infty}\}_{N\ge 1}$ are Feller.
\end{lem}
\begin{proof}
We want to show that $(\Lambda_N^{\infty}f)(\omega)\in C_0(\Omega)$ whenever $f\in C_0(W^N)$. By the density of the Schwartz functions $\mathcal{S}(W^N)$ (smooth with all derivatives decreasing faster than any inverse power of $x$ as $x \to \pm \infty$) in $C_0(W^N)$ it suffices to check this for $f \in \mathcal{S}(W^N)$. The following equality, which is a multidimensional version of the usual Plancherel theorem, is the key tool. It is also the main content of the proof of Theorem 7.7 of Olshanski and Vershik \cite{OlshanskiVershik} and is the equality of displays 7.10 and 7.18 therein. For $\omega$ with $\gamma_2(\omega)>0$,
\begin{align*}
const \times \int_{W^N}^{}\det \left(\phi_{\omega}^{(j-1)}(x_{N+1-i})\right)^N_{i,j=1}\Delta_N(x)f(x)dx=\int_{\mathbb{R}^N}^{}\Delta^2_N(x)F_{\omega}(x_1)\cdots F_{\omega}(x_N)\bar{\hat{f}}(x)dx,
\end{align*}
where $const$ is a positive constant whose exact value will not be important in what follows. Thus, by going to Fourier space we can relate $(\Lambda_N^{\infty}f)(\omega)$ to the functions $F_{\omega}$ for which we have explicit expressions,
\begin{align}\label{Fourier}
(\Lambda_N^{\infty}f)(\omega)=Const \times \int_{\mathbb{R}^N}^{}\Delta^2_N(x)F_{\omega}(x_1)\cdots F_{\omega}(x_N)\bar{\hat{f}}(x)dx.
\end{align}
Furthermore, recall that the Fourier transform $\hat{f}$ of $f \in \mathcal{S}(W^N)$ is still in $\mathcal{S}(W^N)$. Now, observe that (\ref{Fourier})  makes sense for arbitrary $\omega$, even with $\gamma_2(\omega)=0$. Similarly, in order to show continuity in general, first suppose $\omega_n \to \omega$ then, since for any fixed $x\in \mathbb{R}$,  $F_{\omega}(x)$ as a function of $\omega\in \Omega$ is continuous:
\begin{align*}
\Delta^2_N(x)F_{\omega_n}(x_1)\cdots F_{\omega_n}(x_N)\bar{\hat{f}}(x)\to \Delta^2_N(x)F_{\omega}(x_1)\cdots F_{\omega}(x_N)\bar{\hat{f}}(x) \ \ a.e. \ ,
\end{align*}
and thus, by the dominated convergence theorem,
\begin{align*}
(\Lambda_N^{\infty}f)(\omega_n) \to (\Lambda_N^{\infty}f)(\omega).
\end{align*}
Now, in order to show that $(\Lambda_N^{\infty}f)(\omega)$ vanishes as $\omega \to \infty$ we note that $\omega \to \infty$ is actually equivalent to any combination of the following cases, $\gamma_1 \to \pm \infty $ or $\gamma_2 \to \infty$ or $\alpha_1^{\pm} \to \infty$. Observe that any of these possibilities can occur on its own. First, suppose that $\gamma_2 \to \infty$. We see that, since $\hat{f} \in \mathcal{S}(W^N)$, there exists $R<\infty$ such that,
\begin{align*}
\int_{x \notin [-R,R]^N}^{}|\Delta^2_N(x)F_{\omega}(x_1)\cdots F_{\omega}(x_N)\bar{\hat{f}}(x)|dx \lesssim \epsilon.
\end{align*}
And thus,
\begin{align*}
(\Lambda_N^{\infty}f)(\omega) \lesssim \epsilon + \int_{ [-R,R]^N}^{}\left|\Delta^2_N(x)F_{\omega}(x_1)\cdots F_{\omega}(x_N)\bar{\hat{f}}(x)\right|dx \lesssim \epsilon + \left(\int_{-R}^{R}\left|F_{\omega}(y)\right|dy\right)^N.
\end{align*}
But we have,
\begin{align*}
\left|F_{\omega}(y)\right| \le e^{-\frac{\gamma_2}{2}y^2} \ \ \textnormal{in} \ \ [-R,R],
\end{align*}
and so $\left|F_{\omega}(y)\right| \to 0$ as $\gamma_2 \to \infty$ $\forall y \in [-R,R]\backslash \{0\}$ and $\left|F_{\omega}(0)\right|=1$ (in particular bounded). Hence, using the dominated convergence theorem we obtain,
\begin{align*}
\left(\int_{-R}^{R}\left|F_{\omega}(y)\right|dy\right)^N  \to 0 \ \ \textnormal{as}\ \ \gamma_2 \to \infty.
\end{align*}
Of course the integral above can be explicitly calculated in terms of the error function from which the result is evident as well. Now, in order to show that $(\Lambda_N^{\infty}f)(\omega)$ vanishes as $\alpha_1^{\pm} \to \infty$  we follow the same argument, except that now we use the bound,
\begin{align*}
\left|F_{\omega}(y)\right| \le \frac{1}{\sqrt{1+\left(\alpha_1^+y\right)^2}}\frac{1}{\sqrt{1+\left(\alpha_1^-y\right)^2}} \ \ \textnormal{in} \ \ [-R,R],
\end{align*}
and thus $\left|F_{\omega}(y)\right| \to 0$ as either $\alpha_1^{\pm} \to \infty$, $\forall y \in [-R,R]\backslash \{0\}$ from which the claim follows. We finally assume that $\gamma_1 \to \pm\infty$ and take a different approach. First, we write $\Lambda_N^{\infty}f$ as follows, viewing it as a function of $\gamma_1$,
\begin{align*}
(\Lambda_N^{\infty}f)(\gamma_1)=Const \times \int_{\mathbb{R}^N}^{}e^{i\gamma_1x_1+\cdots+i\gamma_1x_N}\left(\prod_{j=1}^{N}e^{-\frac{\gamma_2}{2}x_j^ 2}\right)\prod_{k=1}^{\infty}\left(\prod_{j=1}^{N}\frac{e^{-i\alpha_k^+x_j}}{1-i\alpha_k^+x_j}\right)\prod_{k=1}^{\infty}\left(\prod_{j=1}^{N}\frac{e^{i\alpha_k^-x_j}}{1+i\alpha_k^-x_j}\right)\left(\Delta^2_N(x)\bar{\hat{f}}(x)\right)dx
\end{align*}
and noting that this is exactly \textit{Fourier inversion} of a product which is given in terms of a convolution up to some numerical constant $\tilde{C}$ as follows,
\begin{align*}
(\Lambda_N^{\infty}f)(\gamma_1)=\tilde{C} \times \left(\phi^{\otimes N}_{\gamma_2}*\phi^{\otimes N}_{\alpha^+_1}*\cdots*\phi^{\otimes N}_{\alpha^-_1}*\cdots * g\right)(\gamma_1,\cdots,\gamma_1),
\end{align*}
where $g \in \mathcal{S}(W^N)$ is such that $\hat{g}(\xi)=\Delta^2_N(\xi)\bar{\hat{f}}(\xi)$ .The fact that $(\Lambda_N^{\infty}f)(\gamma_1) \to 0$, as $\gamma_1\to \pm\infty$ now follows, since it is a convolution of $L^1(\mathbb{R}^N)$ functions (in fact it is a Schwartz function). We finally remark that the argument above is essentially just the Riemann-Lebesgue lemma.
\end{proof}

We are finally ready to provide a full proof of Proposition \ref{boundary},
\begin{proof}[Proof of Proposition \ref{boundary}]
By making use of Lemmas \ref{bijection}, \ref{consistency} and \ref{FellerBoundary} we get that the map $\Lambda^{\infty}:\Omega\to \textnormal{Ex}\left(\underset{\leftarrow}{\lim}\mathcal{M}_p\left(W^N\right)\right)$ is a continuous (part of the statement of Lemma \ref{FellerBoundary}) bijection. We obtain that it is actually a Borel isomorphism by Theorem 3.2 in \cite{Mackey}, which states that a Borel one to one map from a standard Borel space onto a subset of a countably generated Borel space is a Borel isomorphism or in this particular setting see Proposition 9.4 of \cite{BorodinOlshanskiErgodic}. This extends to a Borel isomorphism between $\mathcal{M}_p\left(\Omega\right)$ and $\underset{\leftarrow}{\lim}\mathcal{M}_p\left(W^N\right)$ by making use of Theorem 9.1 of \cite{BorodinOlshanskiErgodic} (or more generally the ergodic decomposition theorem for actions of inductively compact groups of Bufetov, namely Theorem 1 in \cite{Bufetov}, see also the proof of Theorem 4.1.3 of \cite{Winkler}). Finally, the Feller assertion follows from Lemmas \ref{Feller1} and \ref{FellerBoundary}.
\end{proof}

\subsection{Markov Processes on the boundary} 

We now, briefly recall the Borodin Olshanski formalism (see in particular Section 2 of \cite{BorodinOlshanski} for detailed proofs), the so called\textit{ method of intertwiners}, for constructing Markov processes on the boundary $\Omega$ ($\Omega$ could in more generality be any locally compact metrizable topological space with a countable base which arises as the $Feller$ boundary of some projective sequence $\{E_N\}_{N\ge 1}$ in the sense described above). 

Hence, let $\{P_N(t)\}_{N\ge 1}$ be a sequence of Markov semigroups on $W^N$ consistent with the \textit{Feller} links above namely,
\begin{align*}
P_{N+1}(t)\Lambda_N^{N+1}=\Lambda_N^{N+1}P_N(t) \ ,\ \forall{t} \ge 0, \ \forall N\ge 1.
\end{align*}

Then, we have the following theorem, proven as Proposition 2.4 in \cite{BorodinOlshanski} (or more precisely a special case of that result applied to our situation),
\begin{thm}\label{BOFormalism}
There exists a unique Markov semigroup $P_{\infty}(t)$ on $\Omega$ so that $\forall N\ge 1$ we have $\forall t \ge 0$,
\begin{align*}
P_{\infty}(t)\Lambda^{\infty}_{N}=\Lambda^{\infty}_{N}P_N(t).
\end{align*}
Moreover, in case the semigroups $P_N(t)$ are Feller then so is $P_{\infty}(t)$.
\end{thm}

\paragraph{Invariant measures} It can be easily seen that, if $\forall N \ge 1$, $\mu_N$ is an invariant measure of $P_N(t)$ and these measures are compatible with the links then the measure $\mu$ on $\Omega$ given by,
\begin{align*}
\mu \Lambda^{\infty}_N=\mu_N,
\end{align*}
is invariant for $P_{\infty}(t)$. If furthermore, we assume that, $\forall N \ge 1$ $\mu_N$ is the \textit{unique} invariant measure of $P_N(t)$ (in such case, compatibility with the links is immediate) then $\mu$ is the \textit{unique} invariant measure for $P_{\infty}(t)$.

\section{Hua-Pickrell measures}\label{sectionhuapickrellmeasures}
In this brief section we define the Hua-Pickrell measures, depending on a complex parameter $s$. We will assume throughout that $\Re(s)>-\frac{1}{2}$. This restriction is necessary in order for the measures to be finite and in particular, we assume that all of them are normalized to have mass 1. We will follow throughout the notation conventions of \cite{BorodinOlshanskiErgodic}. We consider the following measures on $\mathbb{U}(N)$ given by,
\begin{align*}
const \times \det \left((I+U)^{\bar{s}}\right)\det \left((I+U^{-1})^{s}\right) \times dU,
\end{align*}
where $dU$ denotes Haar measure on $\mathbb{U}(N)$. We note that, for $s=0$, this is just the Circular Unitary Ensemble ($CUE$). The projection of this measure on the eigenvalues $(u_1,\cdots,u_N)$ or equivalently the eigenangles, with $u_j=e^{i\theta_j}$ is given by,
\begin{align*}
const \times \prod_{1\le j< k \le N}^{}|u_j-u_k|^2\prod_{j=1}^{N}(1+u_j)^{\bar{s}}(1+\bar{u_j})^s \times d\theta_j.
\end{align*}
Under the inverse Cayley transform $\mathfrak{C}^{-1}$ the corresponding measure on $H(N)$ denoted by $\mathsf{M}_{HP}^{s,N}$ becomes,
\begin{align}\label{MatrixHuaPickrellMeasure}
\mathsf{M}_{HP}^{s,N}(dX)=const \times \det\left((I+iX)^{-s-N}\right)\det \left((I-iX)^{-\bar{s}-N}\right) \times dX,
\end{align}
where $dX$ denotes Lebesgue measure on $H(N)$. Looking at the radial part of $\mathsf{M}_{HP}^{s,N}(dX)$ we get a probability measure on $W^N$ which we will denote by $\mu^{s,N}_{HP}$ and will be referring to as a Hua-Pickrell measure and which is given by,
\begin{align*}
\mu^{s,N}_{HP}(dx)&= const \times \Delta^2_N(x)\prod_{j=1}^{N}(1+ix_j)^{-s-N}(1-ix_j)^{-\bar{s}-N}dx_j\\
&=const \times \Delta^2_N(x)\prod_{j=1}^{N}(1+x^2_j)^{-\Re(s)-N}e^{2\Im(s)Arg(1+ix_j)}dx_j.
\end{align*}
A remarkable property of these measures is that they are coherent with the respect to the links, see for example Proposition 3.1 of  \cite{BorodinOlshanskiErgodic} for a direct proof, 
\begin{align*}
\mu^{s,N+1}_{HP}\Lambda_N^{N+1}=\mu^{s,N}_{HP}.
\end{align*}
This statement will also be derived as Corollary \ref{corconsistency} as a consequence of our intertwining relations between Markov semigroups. We finally denote by $\mu_{HP}^s$ the corresponding measure on $\Omega$. It can be easily seen, that the measures $\mu_{HP}^{s,N}, \ \forall N\ge1$ give rise to determinantal point processes. By an approximation procedure, $\mu_{HP}^s$ does so as well, and this was the main objective of the study of \cite{BorodinOlshanskiErgodic}.
\begin{rmk}
In fact, the situation is a bit more subtle, $\mu_{HP}^s$ gives rise to a determinantal point process in $\mathbb{R}^*$ where $\mathbb{R}^*=\mathbb{R}\backslash \{0\}$ under the so called forgetting map that disregards $\gamma_1(\omega)$ and $\gamma_2(\omega)$ and $\alpha_i^+(\omega),\alpha_j^-(\omega)$ that are zero namely,
\begin{align*}
\omega=\left(\{\alpha^+_i(\omega)\},\{\alpha^+_i(\omega)\},\gamma_1(\omega),\gamma_2(\omega)\right)\mapsto\left(-\alpha_1^-(\omega),-\alpha_2^-(\omega),\cdots, \alpha_2^+(\omega),\alpha_1^+(\omega)\right)\in Conf(\mathbb{R}^*).
\end{align*}
However, in a recent breakthrough Qiu in \cite{Qiu}, has proven that for $s \in \mathbb{R}$ (this covers both the finite and infinite cases) the measure $\mu_{HP}^s$ only charges the subset $\Omega_0\subset \Omega$, defined as the set of all $\omega \in \Omega$ that satisfy:
\begin{align*}
\alpha_i^+(\omega)\ne 0, \alpha_j^-(\omega) \ne 0, \gamma_2(\omega)=0 \textnormal{ and } \gamma_1(\omega)=\lim_{n\to \infty} \left(\sum_{l \in \mathbb{Z}^*}^{}x_{l}(\omega) \textbf{1}_{|x_l(\omega)|>\frac{1}{n^2}}\right)
\end{align*}
where,
\begin{align*}
x_l(\omega)=\begin{cases}
\alpha_l^+(\omega) \textnormal{ if } l >0\\
-\alpha_l^-(\omega) \textnormal{ if } l <0\\
\end{cases}.
\end{align*}
\end{rmk}
\begin{rmk}
For $s=0$, under the forgetting map above and the transform $x \mapsto y=-\frac{1}{\pi x}$ the measure $\mu_{HP}^0$ gives rise to the sine point process, abbreviated $\mathsf{Sine_2}$ here, that is the determinantal point process on $\mathbb{R}$ with correlation kernel given by,
\begin{align*}
K_{\mathsf{Sine_2}}(x,y)=\frac{\sin\left(\pi\left(y-x\right)\right)}{\pi\left(y-x\right)},
\end{align*}
(see Theorem I of \cite{BorodinOlshanskiErgodic}). In particular, the dynamics obtained in Corollary \ref{mainresult} below, under this transform will leave the $\mathsf{Sine_2}$ process invariant.
\end{rmk}

\section{Hua-Pickrell Diffusions}\label{sectiondiffusions}
Before proceeding to define our stochastic dynamics, we remark in passing that, all our dynamical results are valid for any $s\in \mathbb{C}$ and not just for $\Re(s)>-\frac{1}{2}$. So, we begin by considering the one dimensional diffusions that will constitute our basic building blocks. These are strong Markov processes, with continuous sample paths in $\mathbb{R}$, with both $-\infty$ and $\infty$ as natural boundaries and generators given by,
\begin{align*}
L^{(n)}_s=(x^2+1)\frac{d^2}{dx^2}+\left[\left(2-2n-2\Re(s)\right)x+2\Im(s)\right]\frac{d}{dx},
\end{align*}
with invariant/speed measure with density with respect to Lebesgue measure given by,
\begin{align*}
m_s^{(n)}(x)=(1+x^2)^{-\Re(s)-n}e^{2\Im(s)Arg(1+ix)},
\end{align*}
and (the non-exploding) $SDE$ description,
\begin{align*}
dX(t)=\sqrt{2(X^2(t)+1)}dW(t)+\left[\left(2-2n-2\Re(s)\right)X(t)+2\Im(s)\right]dt.
\end{align*}
We will denote by $p_t^{(n),s}(x,y)$ its transition density in $\mathbb{R}$ with respect to Lebesgue measure. 

Moving on, we note that $\Delta_n(x)$ is a positive eigenfunction of $n$ copies of $L_s^{(n)}$-diffusions with eigenvalue denoted by $\lambda_{n,s}$. More precisely,
\begin{lem}
We have $\sum_{i=1}^{n}L_{s,x_i}^{(n)}\Delta_n(x)=\lambda_{n,s}\Delta_n(x)$ where $\lambda_{n,s}=\frac{n(n-1)(-2n+1-3\Re(s))}{3}$.
\end{lem} 
\begin{proof}
I present here an elegant argument suggested by the referee. First, observe that the operator $\sum_{i=1}^{n}L_{s,x_i}^{(n)}$ is symmetric and when applied to a polynomial does not raise the degree. Thus, $\sum_{i=1}^{n}L_{s,x_i}^{(n)}\Delta_n(x)$ is antisymmetric, divisible  by $\Delta_n(x)$ and of the same degree and so actually a multiple of $\Delta_n(x)$. Finally, the coefficient of $x_n^{n-1}x^{n-2}_{n-1}\cdots x_2$ after the application of $\sum_{i=1}^{n}L_{s,x_i}^{(n)}$ gives $\lambda_{n,s}$.
The lemma can also be obtained by iteration of the intertwining relations of the next section.
\end{proof}

As in the introduction, we denote by $P^{s,N}_{HP}(t)$ the Karlin-McGregor semigroup of $N$ $L_s^{(N)}$-diffusions $h$-transformed by $\Delta_N(x)$, namely the semigroup having kernel with $(t,x,y)$ in $(0,\infty) \times \mathring{W}^N \times W^N$ given by,
\begin{align*}
e^{-\lambda_{N,s}t}\frac{\Delta_N(y)}{\Delta_N(x)}\det\left(p^{(N),s}_t(x_i,y_j)\right)^N_{i,j=1}dy.
\end{align*}
The Markov process associated to it, is equivalently given by the unique strong solution, as we see in Lemma \ref{strongsolutionandsemigroup} below, of the system of $SDEs$,
\begin{align}\label{HuaPickrellDiffusion}
dX_i(t)=\sqrt{2(X_i^2(t)+1)}dW_i(t)+\left[\left(2-2N-2\Re(s)\right)X_i(t)+2\Im(s)+\sum_{j\ne i}^{}\frac{2(X^2_i(t)+1)}{X_i(t)-X_j(t)}\right]dt,
\end{align}
where the $\{W_i\}_{i=1}^{N}$ are independent standard Brownian motions. 

\begin{lem}\label{strongsolutionandsemigroup}
The system of $SDEs$ (\ref{HuaPickrellDiffusion}) has a unique strong solution. Moreover, its transition semigroup is given by $P^{s,N}_{HP}(t)$.
\end{lem}

\begin{proof}
We first prove that the system of $SDEs$ (\ref{HuaPickrellDiffusion}) has a unique strong solution with no collisions and no explosions, even if started from a "degenerate" point (when some of the coordinates coincide i.e. there is instant "diffraction" of particles). This follows by applying Theorem 2.2 of \cite{Graczyk} whose conditions we will now proceed to check. In order to apply the aforementioned theorem, one first needs to note that we can write,
\begin{align*}
\sum_{j\ne i}^{}\left(\frac{2(1+x_ix_j)}{x_i-x_j}\right)+(2N-2)x_i=\sum_{j\ne i}^{}\left(\frac{2(1+x^2_i)}{x_i-x_j}\right),
\end{align*}
and thus, one can identify the function $H:\mathbb{R}\times \mathbb{R} \to \mathbb{R}$ in Theorem 2.2 of \cite{Graczyk} as follows,
\begin{align*}
H(x,y)=2(1+xy).
\end{align*}
We now list the conditions of Theorem 2.2 of \cite{Graczyk} in the special case when the functions $\sigma_i\equiv \sigma, b_i\equiv b, H_{ij}\equiv H$ therein do not depend on $i$ and $j$ (note that condition (A5) there is vacuous since $b_i\equiv b$).
\begin{itemize}
\item[(C1)] $|\sigma(x)-\sigma(y)|^2\le \rho \left(|x-y|\right)$ for a function $\rho:\mathbb{R}_+\to \mathbb{R}_+$ such that $\int_{0^+}^{}\rho^{-1}(x)dx=\infty$ and the function $b$ is Lipschitz.
\item[(C2)] There exists a constant $c>0$ such that for all $x,y$:
\begin{align*}
\sigma^2(x)+b(x)x&\le c(1+|x|^2),\\
H(x,y)&\le c(1+|xy|).
\end{align*}
\item[(A1)]For $w<x<y<z$:
\begin{align*}
H(w,z)(y-z)\le H(x,y)(z-w).
\end{align*}
\item[(A2)] There exists a constant $c\ge 0$ such that for all $x,y$:
\begin{align*}
\sigma^2(x)+\sigma^2(y)\le c(x-y)^2+4H(x,y).
\end{align*}
\item[(A3)]  There exists a constant $c\ge 0$ such that for all $x<y<z$:
\begin{align*}
H(x,y)(y-x)+H(y,x)(z-y)\le c(z-y)(z-x)(y-x)+H(x,z)(z-x).
\end{align*}
\item[(A4)] For all $x$:
\begin{align*}
\sigma^2(x)+H(x,x)>0
\end{align*}
or, otherwise for every $y_1,\cdots, y_{N-2}$:
\begin{align*}
b(x)+\sum_{j}^{}\frac{H(x,y_j)}{x-y_j}\textbf{1}(y_j\in \mathbb{R}\backslash \{x\})\neq 0.
\end{align*}
\end{itemize}
 The conditions listed above on the functions:
 \begin{align*}
\sigma(x)=\sqrt{2(1+x^2)}, b(x)=2\Im(s)-2\Re(s)x, H(x,y)=2(1+xy)
 \end{align*}
 can then be checked as follows. 
 
 First of all, condition (C1) holds with $\rho(x)=2x^2$ ($b$ is also obviously Lipschitz). Condition (C2) also clearly holds, by completing the square and after some manipulations we see that any choice of constant $c\ge 2-\Re(s)+|s|$ will do. 
 
Now, condition (A1) requires that for $w<x<y<z$,
\begin{align*}
\frac{1+wz}{z-w}\le \frac{1+xy}{y-x}.
\end{align*}
To see this, define for fixed $w<x<y$ the LHS to be $f(z)=\frac{1+wz}{z-w}$. Since,
\begin{align*}
\frac{d}{dz}f(z)=-\frac{1+w^2}{(z-w)^2},
\end{align*}
and for $z=y$ the inequality $\frac{1+wy}{y-w}\le \frac{1+xy}{y-x}$ is equivalent to $(x-w)(1+y^2)\ge 0$, the statement is immediately seen to be true. 

Moving on to (A2), any choice of a constant $c \ge 3$ will do, since,
\begin{align*}
x^2+y^2+2xy+4 \ge 0.
\end{align*}

For condition (A3), after dividing by 2, we need to find a constant $c \ge 0$ such for $x<y<z$,
\begin{align*}
(1+xy)(y-x)+(1+yz)(z-y)\le c(z-y)(z-x)(y-x)+(1+xz)(z-x).
\end{align*}
Defining $g_c(y)$, for fixed $x<z$ by,
\begin{align*}
g_c(y)= c(z-y)(z-x)(y-x)+(1+xz)(z-x)-(1+xy)(y-x)-(1+yz)(z-y),
\end{align*}
we see that this is a quadratic function in $y$ with zeros at $y=x$ and $y=z$ and leading coefficient $(z-x)(1-c)$. Thus, for $c>1$ we see that $g_c(y)\ge 0$ in $[x,z]$ and the statement follows.

 Condition (A4) obviously holds, since $\sigma^2(x)+H(x,x)>0 \ \forall x \in \mathbb{R}$. 

Hence, the law of $\left((X_1(t),\cdots, X_N(t));t\ge0\right)$ from (\ref{HuaPickrellDiffusion}) is the unique solution to the well-posed martingale problem with generator acting on $C^2_c\left(W^N(\mathbb{R})\right)$ (twice continuously differentiable functions with compact support in $W^N(\mathbb{R})$) given by,
\begin{align*}
\mathsf{L}_s^{(N)}=\Delta^{-1}_N(x)\circ \left(\sum_{i=1}^{N}L^{(N)}_{s,x_i}\right)\circ\Delta_N(x)-\lambda_{N,s}.
\end{align*}
We can now easily observe that, this is exactly given by a Doob's $h$-transform of $N$ one dimensional diffusions; with transition kernel having density with respect to Lebesgue measure in $(0,\infty) \times \mathring{W}^N \times W^N$ given by,
\begin{align*}
e^{-\lambda_{N,s}t}\frac{\Delta_N(y)}{\Delta_N(x)}\det\left(p^{(N),s}_t(x_i,y_j)\right)^N_{i,j=1},
\end{align*}
where $p^{(N),s}_t(x,y)$ is the \textit{Feller} transition density of the one dimensional diffusion process with generator $L_s^{(N)}$ with two natural boundaries.
\end{proof}

We now give a direct and rather technical proof that the semigroups are Feller. A much neater argument is given in Section \ref*{sectionmatrixprocess}, however one needs to introduce a quite non-trivial matrix valued stochastic process, having (\ref{HuaPickrellDiffusion}) as its eigenvalue evolution. The matrix process is in some sense better behaved from an $SDE$ point of view, so we can appeal to existing results in the literature.

\begin{lem}\label{FellerSemigroup}
The semigroups $P^{s,N}_{HP}(t)$ are Feller in the sense that $\forall f \in C_0(W^N)$ we have,
\begin{align*}
& P^{s,N}_{HP}(t)f\in C_0(W^N), \ \ \forall t>0 \ ,\\
& \lim_{t\to 0}P^{s,N}_{HP}(t)f=f.
\end{align*}
\end{lem}
\begin{proof}
For each $(x_1\le \cdots \le x_N) \in W^N$ the continuity of $t\mapsto \mathbb{E}_{(x_1,\cdots,x_N)}\left[f\left(X_1(t),\cdots,X_N(t)\right)\right]$ with $f\in C_0$ follows from the fact that $\left(X_1(t),\cdots,X_N(t);t \ge 0\right)$ is the unique strong solution of the system of $SDEs$ even if started from the diagonals. More specifically, this follows by the almost sure continuity in $t$ of $(X_1(t),\cdots X_N(t);t\ge 0)$ (see statement of Theorem 5.1 of \cite{Graczyk}).

For the fact that $P^{s,N}_{HP}(t)f \in C_0$ if $f\in C_0$, first pick $R$ such that $|f(x_1, \cdots , x_N)| \le \epsilon$ for $(x_1\le \cdots \le x_N)\notin [-R,R]^N$ and let us write for $(x_1, \cdots , x_N) \in W^N$,
\begin{align*}
|\mathbb{E}_{(x_1,\cdots,x_N)}\left[f\left(X_1(t),\cdots,X_N(t)\right)\right]| &\le \mathbb{E}_{(x_1,\cdots,x_N)}\left[|f\left(X_1(t),\cdots,X_N(t)\right)|\mathbf{1}\left(X(t)\in [-R,R]^N\right)\right]+\\
&\mathbb{E}_{(x_1,\cdots,x_N)}\left[|f\left(X_1(t),\cdots,X_N(t)\right)|\mathbf{1}\left(X(t)\notin [-R,R]^N\right)\right]\\
\le \|f\|_{\infty}\mathbb{P}_{(x_1,\cdots,x_N)}\left(X(t)\in [-R,R]^N\right)&+\epsilon \mathbb{P}_{(x_1,\cdots,x_N)}\left(X(t)\notin [-R,R]^N\right),
\end{align*}
and also for $(x_1, \cdots , x_N),(y_1, \cdots , y_N) \in W^N$ with $P^{s,N}_{HP}(t)\left((x_1,\cdots,x_N),dz\right)$ being the law of $(X_1(t),\cdots,X_N(t))$ if $(X_1(0),\cdots,X_N(0))=(x_1,\cdots,x_N)$,
\begin{align*}
&|\mathbb{E}_{(x_1,\cdots,x_N)}\left[f\left(X_1(t),\cdots,X_N(t)\right)\right]-\mathbb{E}_{(y_1,\cdots,y_N)}\left[f\left(X_1(t),\cdots,X_N(t)\right)\right]|\\
&\le\|f\|_{\infty}\int_{W^N\cap[-R,R]^N}^{}\left|P^{s,N}_{HP}(t)\left((x_1,\cdots,x_N),dz\right)-P^{s,N}_{HP}(t)\left((y_1,\cdots,y_N),dz\right)\right|+2\epsilon.
\end{align*}
Both assertions (vanishing at infinity and continuity) will follow immediately by the use of the dominated convergence theorem and the estimates on the transition density and its derivatives in the backwards variables $\partial_x^{(i)}p^{(N),s}_t(x,y)$ (for $i \ge 0$) to be presented shortly.

To be more concrete and in order to ease notation, let us first consider the most singular case $x_1=\cdots=x_N=x$, the arguments for the others are analogous and will be explained at the end of this proof. First, note that the law of $(X_1(t),\cdots,X_N(t))$ started from $(x,\cdots,x)$ is governed by (this being well-posed is justified by the estimates presented below),
\begin{align*}
\underset{x_1,\cdots,x_N \to x \mathbf{1}}{\lim}e^{-\lambda_N t}\frac{\det\left(y_i^{j-1}\right)^N_{i,j=1}}{\det\left(x_i^{j-1}\right)^N_{i,j=1}}\det \left(p^{(N),s}_t(x_i,y_j)\right)^N_{i,j=1}dy=e^{-\lambda_N t}\Delta_N(y) \det \left(\partial^{(i-1)}_xp^{(N),s}_t(x,y_j)\right)_{i,j=1}^N dy.
\end{align*}

Hence, by expanding the determinant, we have the bound,
\begin{align*}
e^{-\lambda_Nt} \int_{W^N\cap [-R,R]^N}^{}\Delta_N(y)\det \left(\partial^{(i-1)}_xp^{(N),s}_t(x,y_j)\right)_{i,j=1}^N dy \lesssim C(N,t,R) \prod_{i=1}^{N}\int_{-R}^{R}|\partial^{(i-1)}_xp^{(N),s}_t(x,z)|dz.
\end{align*}

From now on, to ease notation further, we write $p_t(x,y)$ for the transition density with respect to Lebesgue measure of the SDE in $\mathbb{R}$,
\begin{align*}
dX(t)=\sqrt{2(X^2(t)+1)}dW(t)+(\beta X(t)+\gamma)dt,
\end{align*}
where $\beta$ and $\gamma$ are arbitrary (real) constants. We make the following smooth change of variables (in order to obtain bounded coefficients),
\begin{align*}
Y(t)=arsinh(X(t))=\log\left(X(t)+\sqrt{1+X^2(t)}\right).
\end{align*}
Hence, with $y=f(x)=arsinh(x)$ we have $f'(x)=\frac{1}{\sqrt{1+x^2}}$ and $f''(x)=-\frac{x}{(1+x^2)^{\frac{3}{2}}}$ and by applying Ito's formula we obtain,
\begin{align*}
dY(t)=\sqrt{2}dW(t)+\left[(\beta-1)tanh(Y(t))+\gamma sech(Y(t))\right]dt.
\end{align*}
or equivalently $Y(t)$ is a diffusion in $\mathbb{R}$ with generator $A$ given by,
\begin{align*}
A=\frac{d^2}{dx^2}+\left[(\beta-1)tanh(x)+\gamma sech(x)\right] \frac{d}{dx}.
\end{align*}
Now, note that the coefficients are smooth with all their derivatives bounded and (obviously) the diffusion coefficient is uniformly elliptic. Thus, if we let $q_t(z,w)$ denote the transition density of $(Y(t);t \ge 0)$, from Theorem 3.3.11 of \cite{Stroock}, we have for $i\ge 0$ with some constant $C_i$ (depending on the ellipticity constant and the derivatives of the coefficients) the following bound,
\begin{align*}
|\partial_z^{(i)}q_t(z,w)| \le \frac{C_i}{1 \wedge t^{\frac{i+1}{2}}}\exp \left(-\left(C_it- \frac{(z-w)^2}{C_i t}\right)^-\right).
\end{align*}
By applying the change of variables, the original kernel $p_t(x,y)$ for $X(t)$ is given by,
\begin{align*}
p_t(x,y)=q_t(f(x),f(y)) \partial_yf(y)\ \ \textnormal{ where } \ \ f(x)=arsinh(x).
\end{align*}
Now, making use of Faa-Di Bruno's formula, we obtain,
\begin{align*}
\partial_x^{(i)}p_t(x,y)= \sum_{}^{}\frac{i!}{k_1! \cdots k_i!}\partial^{(k)}_{f(x)}q_t(f(x),f(y))\prod_{j=1}^{i}\left(\frac{\partial_x^{(j)}f(x)}{j!}\right)^{k_j}\partial_yf(y),
\end{align*}
where $k=k_1+\cdots +k_i$ and the sum is over $k_1,\cdots,k_i$ such that $k_1+2k_2+\cdots+ik_i=i$. 

This is a finite sum and applying the triangle inequality, we will arrive at some sufficient bound but we can in fact get the leading order terms for each of the exponentials. Observe that for $j \ge 1$,
\begin{align*}
|\partial_x^{(j)}f(x)| \le \frac{c_j}{\left(1+x^2\right)^{\frac{j}{2}}}+o\left(\frac{1}{\left(1+x^2\right)^{\frac{j}{2}}}\right).
\end{align*}
Hence, making use of the fact $k_1+2k_2+\cdots+ik_i=i$ we get,
\begin{align*}
|\partial_x^{(i)}p_t(x,y)| \lesssim \left(\frac{1}{\sqrt{1+x^2}}\right)^{i}\left(\frac{1}{\sqrt{1+y^2}}\right)\sum_{j=0}^{i}c(j,i,t)\exp \left(-\left(C_jt- \frac{(arsinh(x)-arsinh(y))^2}{C_j t}\right)^-\right) + \mathsf{l}.\mathsf{o}.\mathsf{t},
\end{align*}
where $\mathsf{l}.\mathsf{o}.\mathsf{t}$ stands for lower order terms. By the continuity of $ x \mapsto \partial_x^{(i)}p_t(x,y)$, the estimate above and the dominated convergence theorem the Feller property follows.

We will now treat the more general case when some of the points $(x_1^{(n)},\cdots, x_N^{(n)})$, not necessarily all, can come together as they go to $\infty$ with $n \to \infty$. First, we write:
\begin{align*}
P^{s,N}_{HP}(t)(x_1,\cdots,x_N;y_1,\cdots,y_N)=e^{-\lambda_{N,s}t}\frac{\Delta_N(y)}{\Delta_N(x)}F_t(x_1,\cdots,x_N;y_1,\cdots,y_N),
\end{align*}
where,
\begin{align*}
F_t(x_1,\cdots,x_N;y_1,\cdots,y_N)=\det\left(p^{(N),s}_t(x_i,y_j)\right)^N_{i,j=1}.
\end{align*}
We can then split $(x_1^{(n)},\cdots, x_N^{(n)})$ into $m$ blocks $(x^{(n)}_{i_1+\cdots+i_{j-1}+1}, \cdots, x^{(n)}_{i_1+\cdots+i_j})$, with $i_1+\cdots+i_m=N$ and $i_0=0$ such that $|x^{(n)}_{i_1+\cdots+i_{j}}-x^{(n)}_{i_1+\cdots+i_{j}+1}|\ge \textnormal{Const}$ for $j=1,\cdots,m$ uniformly in $n$. 

From now on we will suppress the dependence of $F$ on $t,y_1,\cdots,y_N$ and write $F(x_1,\cdots,x_N)$. Note that, $(x_1^{(n)},\cdots, x_N^{(n)})\to \infty$ if and only if at least one of $x^{(n)}_N \to \infty$ or $x^{(n)}_1\to -\infty$ happens and without loss of generality we assume that $x^{(n)}_1 \to -\infty$. The problematic singular terms coming from the Vandermonde determinant $\Delta_N(x)$ are of course:
\begin{align*}
\frac{1}{\prod_{i_1+\cdots+i_{j-1}+1\le l_1 < l_2 \le i_1+\cdots+i_{j}}^{}(x^{(n)}_{l_2}-x^{(n)}_{l_1})}
\end{align*}
which blow up as $n \to \infty$. The crux is that these singularities are cancelled out by vanishing terms coming from $F_t(x_1,\cdots,x_N;y_1,\cdots,y_N)$.

We begin by applying the mean-value theorem (MVT) to the first block and we will suppress dependence on $n$ from now on. To ease notation write $F(x_1,\cdots,x_N)=\tilde{F}(x_1,\cdots,x_k)$ where $k=i_1$ and we write $\partial_l$ for the derivative with respect to the $l^{th}$ variable. Then, since $\tilde{F}\left(x_1,\cdots,x_{k-1},x_{k-1}\right)=0$, we have for some $\xi^{1}_{k}$ such that $x_{k-1}< \xi^{1}_{k}<x_{k}$:
\begin{align*}
\tilde{F}(x_1,\cdots,x_{k-1},x_k)=(x_k-x_{k-1})\partial_k\tilde{F}\left(x_1,\cdots,x_{k-1},\xi_k^1\right).
\end{align*}
Now write $\xi_i^0=x_i$. Applying the MVT $(k-2)$ more times we obtain that for some $(\xi_2^1,\cdots,\xi^1_k)$ satisfying $\xi_1^0 < \xi_2^1 < \xi_2^0 < \cdots < \xi_k^1 < \xi_k^0$:
\begin{align*}
\tilde{F}(x_1,\cdots,x_{k-1},x_k)=\prod_{i=1}^{k-1}\left(\xi^{0}_{i+1}-\xi_i^0\right)\partial_2\cdots \partial_k \tilde{F}(\xi_1^0,\xi_2^1,\cdots,\xi^1_{k-1},\xi^{1}_k).
\end{align*}
Iterating this procedure we finally get:
\begin{align*}
\tilde{F}(x_1,\cdots,x_{k-1},x_k)=\prod_{l=0}^{k-2}\prod_{i=l+1}^{k-1}(\xi^l_{i+1}-\xi_i^l)\partial_2\partial_3^{2}\cdots \partial^{k-2}_{k-1}\partial^{k-1}_k \tilde{F}(\xi_1^0,\xi_2^1,\cdots,\xi^{k-2}_{k-1},\xi^{k-1}_k),
\end{align*}
for some $\xi_i^l$, $l=0,\cdots, k-2$, $i=l+1,\cdots,k$ such that:
\begin{align*}
\xi_{l+1}^l<\xi_{l+2}^{l+1}<\xi_{l+2}^l<\cdots <\xi_k^{l+1}<\xi_k^l.
\end{align*}
By the interlacing constraints above we observe that, for all $l=0,\cdots,k-2$ and $i=l+1,\cdots,k-1$:
\begin{align*}
\xi^l_{i+1}-\xi_{i}^l\le x_{i+1}-x_{i-l}=\xi^0_{i+1}-\xi_{i-l}^0.
\end{align*}
Thus, the following ratio is bounded:
\begin{align*}
\frac{\prod_{l=0}^{k-2}\prod_{i=l+1}^{k-1}(\xi^l_{i+1}-\xi_i^l)}{\prod_{1 \le i < j \le k}^{}(x_{j}-x_{i})} \le 1.
\end{align*}
In particular it will be uniformly bounded in $n$ when the $x$'s depend on $n$. Now we can obviously apply the argument above to each single block $(x^{(n)}_{i_1+\cdots+i_{j-1}+1}, \cdots, x^{(n)}_{i_1+\cdots+i_j})$ for $j=1, \cdots, m$. Then the result follows by the uniform bounds on the transition kernel and its derivatives $\partial_x^{(j)}p_t(x,y)$; in particular we need bounds for the first $\underset{j=1,\cdots,m}{\sup}i_j-1$ derivatives.\

\end{proof}

We now arrive at the following proposition, which makes explicit the relation between the Hua-Pickrell measures and the Hua-Pickrell diffusions.
\begin{prop}\label{invariant}
Let $\Re(s)>-\frac{1}{2}$. Then the probability measure $\mu^{s,N}_{HP}$ is the unique invariant measure of $P^{s,N}_{HP}(t)$.
\end{prop}

\begin{proof}
By making use of the reversibility of $p^{(N),s}_t(x,y)$ with respect to $m_s^{(N)}(x)$ and the fact that $\Delta_N$ is an eigenfunction of the sub-Markov Karlin-McGregor semigroup with eigenvalue $e^{\lambda_{N,s}t}$, we can obtain the invariance of $\mu^{s,N}_{HP}$ by $P^{s,N}_{HP}(t)$ as follows (here $const$ denotes the same normalization constant in all equalities),
\begin{align*}
&\int_{\mathring{W}^N}^{}e^{-\lambda_{N,s}t}\frac{\Delta_N(y)}{\Delta_N(x)}\det\left(p^{(N),s}_t(x_i,y_j)\right)^N_{i,j=1}\times const \times \Delta^2_N(x)\prod_{j=1}^{N}(1+x^2_j)^{-\Re(s)-N}e^{2\Im(s)Arg(1+ix_j)}dx=\\
&=const \times \Delta_N(y)\prod_{j=1}^{N}(1+y^2_j)^{-\Re(s)-N}e^{2\Im(s)Arg(1+iy_j)}e^{-\lambda_{N,s}t}\int_{W^N}^{}\det\left(p^{(N),s}_t(y_i,x_j)\right)^N_{i,j=1}\Delta_N(x)dx\\
&=const \times \Delta_N(y)\prod_{j=1}^{N}(1+y^2_j)^{-\Re(s)-N}e^{2\Im(s)Arg(1+iy_j)}e^{-\lambda_{N,s}t}e^{\lambda_{N,s}t}\Delta_N(y).
\end{align*}
Now, by the regularity of the transition kernel $e^{-\lambda_{N,s}t}\frac{\Delta_N(y)}{\Delta_N(x)}\det\left(p^{(N),s}_t(x_i,y_j)\right)^N_{i,j=1}$ we show that actually $\mu^{s,N}_{HP}$ is the \textit{unique} invariant probability measure of $P^{s,N}_{HP}(t)$. Namely, suppose we had at least two different invariant probability measures, then we would have at least two distinct \textit{ergodic} ones which have to be mutually singular (see Lemma 2.10 and Theorem 2.11 of \cite{Eberle}). Now, since $\tau=\inf \{t>0: \exists \ \ 1 \le i < j \le N  \ \textnormal{such that} \ X_i(t)=X_j(t) \}=\infty$ almost surely (the system of $SDEs$ (\ref{HuaPickrellDiffusion}) has no collisions or equivalently never hits a diagonal) then any invariant measure $\mu$ of $P_{HP}^{s,N}(t)$ does not charge $\partial W^N$. Hence, if $\mu_1, \mu_2$ are two (distinct) ergodic measures there exists some Borel set $A_1$ so that $A_1 \not \subset \partial W^N$ and,
\begin{align}\label{mutuallysingular}
\mu_1\left(A_1\right)=1 \ \textnormal{and} \ \mu_2\left(A_1\right)=0.
\end{align}
Moreover, note that $A_1$ must have positive Lebesgue measure denoted $Leb(A_1)>0$ for otherwise by the invariance of $\mu_1$ we would have (since $P_{HP}^{s,N}(t)$ has a density $P_{HP}^{s,N}(t)(x,y)$ with respect to Lebesgue),
\begin{align*}
\mu_1\left(A_1\right)=\int_{W^N}^{}\mu_1(dx) \int_{A_1}^{}P_{HP}^{s,N}(t)(x,y)dy=0.
\end{align*}
But on the other hand, since we have the following fundamental \textit{strict total positivity} fact,
\begin{align*}
e^{-\lambda_{N,s}t}\frac{\Delta_N(y)}{\Delta_N(x)}\det\left(p^{(N),s}_t(x_i,y_j)\right)^N_{i,j=1}>0 \ ,\ \forall \left(t,x,y\right)\in (0,\infty) \times \mathring{W}^N \times \mathring{W}^N,
\end{align*}
which is exactly (a particular case of) the statement of Theorem 4 of \cite{KarlinMcGregorCoincidence} or see also Problem 6 and its solution on pages 158-159 of \cite{ItoMckean}, we obtain that for any Borel set $\mathcal{A}$ such that $\mathcal{A} \not \subset \partial W^N$ and $Leb(\mathcal{A})>0$,
\begin{align*}
f^{(N)}_{\mathcal{A},t}(x)=\int_{\mathcal{A}}^{}e^{-\lambda_{N,s}t}\frac{\Delta_N(y)}{\Delta_N(x)}\det\left(p^{(N),s}_t(x_i,y_j)\right)^N_{i,j=1}dy>0 \ ,\ \forall x \in \mathring{W}^N.
\end{align*}
Thus, by the invariance of $\mu_i$ for $i=1,2$ and the fact that they do not charge $\partial W^N$, we get,
\begin{align*}
\mu_i\left(\mathcal{A}\right)=\int_{W^N}^{}\mu_i(dx) \int_{\mathcal{A}}^{}e^{-\lambda_{N,s}t}\frac{\Delta_N(y)}{\Delta_N(x)}\det\left(p^{(N),s}_t(x_i,y_j)\right)^N_{i,j=1}dy=\int_{W^N}^{}\mu_i(dx)f^{(N)}_{\mathcal{A},t}(x)>0,
\end{align*}
which contradicts (\ref{mutuallysingular}) and thus we obtain uniqueness.
\end{proof}

\section{Intertwinings and Boundary Feller process}\label{sectionintertwinings}

In this section, we prove the main result of this paper, proven as Corollary \ref{mainresult} below. In order to proceed, we first need to recall one of the main results of \cite{InterlacingDiffusions} that we require here (we give a self-contained proof in Section \ref{appendix}). As we will see in the proof of Theorem \ref*{intertwiningtheorem} below, the additional contribution of this paper, other than the quite non-trivial technical work of proving that all Markov kernels and semigroups are Feller; is a rather simple observation regarding one-dimensional diffusion generators, which is actually what made it clear to the author that the method of intertwiners could be applied in this setting.

We begin by defining the \textit{dual} Hua-Pickrell diffusion in $\mathbb{R}$, with infinitesimal generator denoted by $\widehat{L^{(n)}_s}$ given by,
\begin{align*}
\widehat{L^{(n)}_s}&=(x^2+1)\frac{d^2}{dx^2}+\left[2x-\left(2-2n-2\Re(s)\right)x-2\Im(s)\right]\frac{d}{dx}\\
&=(x^2+1)\frac{d^2}{dx^2}+\left[\left(2n+2\Re(s)\right)x-2\Im(s)\right]\frac{d}{dx},
\end{align*}
 and where, both $-\infty$ and $+\infty$ are natural boundary points. The corresponding (non-exploding) $SDE$ is given by,
\begin{align*}
dX(t)=\sqrt{2(X^2(t)+1)}dW(t)+\left[\left(2n+2\Re(s)\right)X(t)-2\Im(s)\right]dt
\end{align*}
and the speed measure $\hat{m}_s^{(n)}$ with density with respect to Lebesgue measure given by,
\begin{align*}
\hat{m}_s^{(n)}(x)=(1+x^2)^{\Re(s)+n-1}e^{-2\Im(s)Arg(1+ix)}.
\end{align*}
Propositions 2.15 and 2.16, more precisely display (26) of \cite{InterlacingDiffusions} give the intertwining relation $\forall t > 0, N \ge 1$,
\begin{align}\label{intermediateintertwining}
\mathcal{P}^{(N+1)}_{s}(t)\Lambda_{N,N+1}=\Lambda_{N,N+1}\hat{\mathcal{P}}^{(N)}_{s}(t),
\end{align}
where $\mathcal{P}^{(N+1)}_{s}(t)$ is the sub-Markov Karlin-McGregor semigroup associated to $N+1$ $L_s^{(N+1)}$-diffusions killed when they intersect or equivalently the semigroup with kernel in $W^{N+1}$ given by,
\begin{align*}
\det\left(p^{(N+1),s}_t(x_i,y_j)\right)^{N+1}_{i,j=1}dy.
\end{align*}
Similarly, $\hat{\mathcal{P}}^{(N)}_{s}(t)$ is the sub-Markov Karlin-McGregor semigroup associated to $N$ $\widehat{L_s^{(N+1)}}$-diffusions having kernel (where we denote by $\hat{p}^{(N+1),s}_t$ the transition kernel of a single $\widehat{L_s^{(N+1)}}$-diffusion process),
\begin{align*}
\det\left(\hat{p}^{(N+1),s}_t(x_i,y_j)\right)^{N}_{i,j=1}dy,
\end{align*}
and $\Lambda_{N,N+1}$ is the, not yet normalized, positive kernel,
\begin{align*}
\Lambda_{N,N+1}(x,dy)=\prod_{i=1}^{N}\hat{m}_s^{(N+1)}(y_i)\boldmath{1}(y\in W^{N,N+1}(x))dy.
\end{align*}

For completeness, we will give a self-contained proof of (\ref{intermediateintertwining}) in Section \ref{appendix}. We are now ready to state and prove the key theorem behind the construction:

\begin{thm} \label{intertwiningtheorem} Let $N \ge 1$ and $f \in C_0\left(W^N\right)$ then $\forall t \ge 0$,
\begin{align}\label{intertwining}
P^{s,N+1}_{HP}(t)\Lambda^{N+1}_Nf=\Lambda^{N+1}_NP^{s,N}_{HP}(t)f.
\end{align}
\end{thm}

\begin{proof}
The proof hinges on the following simple observation regarding one-dimensional diffusion operators: namely an easy calculation gives that the function $\left(\hat{m}_s^{(N+1)}\right)^{-1}(x)$ is a positive eigenfunction of $\widehat{L^{(N+1)}_s}$ with eigenvalue $c_{N,s}=-2N-2\Re(s)$ and the $h$-transform of $\widehat{L^{(N+1)}_s}$ by this eigenfunction is the $L^{(N)}_s$-diffusion. In symbols, at the infinitesimal level:
\begin{align*}
\hat{m}_s^{(N+1)}(x)\circ\widehat{L^{(N+1)}_s}\circ\left(\hat{m}_s^{(N+1)}\right)^{-1}(x)-c_{N,s}=L_s^{(N)}
\end{align*}
and at the level of transition densities, for $t>0$:
\begin{align*}
e^{-c_{N,s} t}\hat{p}_t^{(N+1),s}(x,y)\frac{\left(\hat{m}_s^{(N+1)}\right)^{-1}(y)}{\left(\hat{m}_s^{(N+1)}\right)^{-1}(x)}=p_t^{(N),s}(x,y).
\end{align*}
For $y \in W^N$ consider $h_{N,s}(y)=\prod_{i=1}^{N}\left(\hat{m}_s^{(N+1)}\right)^{-1}(y_i)\Delta_N(y)$ and observe that:
\begin{align*}
\left(\Lambda_{N,N+1}h_{N,s}\right)(x)=\frac{1}{N!}\Delta_{N+1}(x)
\end{align*}
and so for $x\in \mathring{W}^{N+1}$:
\begin{align*}
\left[\frac{1}{\left(\Lambda_{N,N+1}h_{N,s}\right)(x)}\Lambda_{N,N+1}\circ h_{N,s}(y)\right](x,dy)=\Lambda_{N}^{N+1}(x,dy).
\end{align*}
Moreover, note that $\lambda_{N+1,s}=\lambda_{N,s}+Nc_{N,s}$. 

Thus, performing an $h$-transform of the right hand side of ($\ref{intermediateintertwining}$) by $e^{-\left(\lambda_{N,s}+Nc_{N,s}\right)t}h_{N,s}(y)$, which in probabilistic terms corresponds to transforming the $N$ $\widehat{L^{(N+1)}_s}$-diffusions into $N$ $L^{(N)}_s$-diffusions and conditioning those by the Vandermonde determinant $\Delta_N(y)$ (and analogously for the left hand side), we obtain the following equality of Markov kernels for $t>0$ and $x \in \mathring{W}^{N+1}$,
\begin{align*}
(P^{s,N+1}_{HP}(t)\Lambda^{N+1}_N)(x,dy)=(\Lambda^{N+1}_NP^{s,N}_{HP}(t))(x,dy).
\end{align*}
Now, by using the Feller property of the kernels involved (Lemma \ref{Feller1} and Lemma \ref{FellerSemigroup}), we can extend this to $t \ge 0$ and $x \in W^{N+1}$ and obtain the statement of the theorem.
\end{proof}

Making use of Proposition \ref{invariant}, we immediately get the following corollary,

\begin{cor}\label{corconsistency}
Let $\Re(s)>-\frac{1}{2}$ then the Hua-Pickrell measures are consistent with the links,
\begin{align*}
\mu^{s,N+1}_{HP}\Lambda_N^{N+1}=\mu^{s,N}_{HP}.
\end{align*}
\end{cor}

Finally, using Theorem \ref{intertwiningtheorem} above and Theorem \ref{BOFormalism} we readily get,

\begin{cor}\label{mainresult}
There exists a unique Feller semigroup $P^{s,\infty}_{HP}(t)$ on $\Omega$ that is consistent with the semigroups $\{P_N(t)\}_{N\ge 1}$, so that for $f \in C_0\left(W^N\right)$,
\begin{align*}
P^{s,\infty}_{HP}(t)\Lambda^{\infty}_Nf=\Lambda^{\infty}_NP^{s,N}_{HP}(t)f \ ,\ \forall t\ge0, \ \forall N \ge 1.
\end{align*}
Moreover, if $\Re(s)>-\frac{1}{2}$ the measure $\mu^s_{HP}$ is its unique invariant measure.
\end{cor}

\subsection{Approximation of processes on the boundary}
For any Feller process encountered below, taking values in a locally compact metrizable separable space $\mathcal{X}$, we assume that we are always dealing with its cadlag modification in the space $\mathsf{D}\left(\mathbb{R}_+,\mathcal{X}\right)$, of right continuous functions with left limits. In order to describe the approximation procedure, we begin by recalling some of the setup. Suppose $\{M_N\}_{N\ge 1}$ is a sequence of coherent probability measures on $\{W^N\}_{N\ge 1}$,
\begin{align*}
M_{N+1}\Lambda_{N}^{N+1}=M_N \ ,\forall N \ge 1,
\end{align*}
and let $M$ denote the corresponding measure on $\Omega$. We can embed $W^N$ into $\Omega$ as follows, by defining for $x^{(N)} \in W^N$,
\begin{align*}
\alpha_i^+\left(x^{(N)}\right)&=\begin{cases}
\frac{\max\{x^{(N)}_{N+1-i},0\}}{N}  \ & i=1,\cdots, N\\
0 & i=N+1,N+2,\cdots
\end{cases},\\
\alpha_i^-\left(x^{(N)}\right)&=\begin{cases}
\frac{\max\{-x^{(N)}_{i},0\}}{N}  \ & i=1,\cdots, N\\
0 & i=N+1,N+2,\cdots
\end{cases},\\
\gamma_1\left(x^{(N)}\right)&=\sum_{i=1}^{\infty}\alpha_i^+\left(x^{(N)}\right)-\sum_{i=1}^{\infty}\alpha_i^-\left(x^{(N)}\right)=\frac{x_1^{(N)}+\cdots+x_N^{(N)}}{N},\\
\delta\left(x^{(N)}\right)&=\sum_{i=1}^{\infty}\left(\alpha_i^+\left(x^{(N)}\right)\right)^2+\sum_{i=1}^{\infty}\left(\alpha_i^-\left(x^{(N)}\right)\right)^2=\frac{\left(x_1^{(N)}\right)^2+\cdots+\left(x_N^{(N)}\right)^2}{N^2}.
\end{align*}
We will denote these embeddings by $\mathfrak{r}_N: W^N \hookrightarrow \Omega$ and hence we can view each $M_N$ as a probability measure on $\Omega$ under the pushforward $\left(\mathfrak{r}_N\right)_*M_N$. Then, from Sections 4 and 5 of \cite{BorodinOlshanskiErgodic}, see also Section 2.1 of \cite{Qiu}, $M$ is the measure on $\Omega$ corresponding to the \textit{coherent family} $\{M_N\}_{N\ge 1}$ if and only if the following convergences in distribution hold as $N \to \infty$,
\begin{align*}
\alpha_i^+\left(x^{(N)}\right)&\overset{d}{\longrightarrow} \alpha_i^+\left(\omega\right)\ , \forall i \ge 1,\\
\alpha_i^-\left(x^{(N)}\right)&\overset{d}{\longrightarrow} \alpha_i^-\left(\omega\right)\ , \forall i \ge 1,\\
\gamma_1\left(x^{(N)}\right)&\overset{d}{\longrightarrow} \gamma_1\left(\omega\right),\\
\delta\left(x^{(N)}\right)&\overset{d}{\longrightarrow} \delta\left(\omega\right),
\end{align*}
where $x^{(N)}$ is sampled according to $M_N$ and $\omega$ according to $M$. And in such a case, as before, we write,
\begin{align*}
\gamma_2\left(\omega\right)=\delta\left(\omega\right)-\sum_{i=1}^{\infty}\left(\alpha_i^+\left(\omega\right)\right)^2-\sum_{i=1}^{\infty}\left(\alpha_i^-\left(\omega\right)\right)^2.
\end{align*}

Now, consider a family of Feller semigroups $\{P_N(t);t\ge 0 \}_{N\ge 1}$ consistent with the links $\Lambda_N^{N+1}$ and let $\left(X^{(N)}(t);t\ge 0\right)$ denote a realization of the corresponding Markov processes. Moreover, let $P_{\infty}(t)$ be the semigroup on $\Omega$ obtained by the method of the intertwiners and denote a realization of this by $\left(\mathsf{X}_{\infty}(t);t\ge 0\right)$. Note that, we can of course, embed $\mathsf{D}\left(\mathbb{R}_+,W^N\right)$ into $\mathsf{D}\left(\mathbb{R}_+,\Omega\right)$, in the obvious way and by abusing notation we write $\alpha_i^+\left(X^{(N)};t\right), \alpha_i^-\left(X^{(N)};t\right), \gamma_1\left(X^{(N)};t\right), \delta\left(X^{(N)};t\right)$ for this. Moreover, we still denote these embeddings by $\mathfrak{r}_N$. We then have the following proposition.

\begin{prop}\label{propositionapproximation}
For each $N\ge 1$, let $\left(X^{(N)}(t);t \ge 0\right)$ be Feller processes in $W^N$ that are consistent with the links $\Lambda_N^{N+1}$ $\forall N \ge 1$. Denote by $\left(\mathsf{X}_{\infty}(t);t\ge 0\right)$ the Feller-Markov process on $\Omega$ obtained by the method of the intertwiners  and also let as before $\left(\bar{X}^{(N)}(t);t \ge 0\right)=\left(\mathfrak{r}_N\left(X^{(N)}\right)(t);t \ge 0\right)$. Finally, assume that  $\{\mu_N\}_{N\ge 1}$ is a consistent family of probability measures with corresponding measure $\mu$ on $\Omega$. Then, if $\forall {N}\ge 1$ $\bar{X}^{(N)}(0)\overset{d}{=}\left(\mathfrak{r}_N\right)_*\mu_N$ and $\mathsf{X}_{\infty}(0)\overset{d}{=}\mu$ we have for any fixed $t\ge 0$,
\begin{align*}
\bar{X}^{(N)}(t)\overset{d}{\longrightarrow}\mathsf{X}_{\infty}(t) \ \textnormal { as }  \ N \to \infty,
\end{align*}
or equivalently,
\begin{align*}
\mathsf{X}_{\infty}(t)\overset{d}{=}\underset{N\to \infty}{w\textendash \lim}\ \left(\mathfrak{r}_N\right)_*\left(\mu_NP_N(t)\right),
\end{align*}
where $w\textendash \lim$ denotes the weak limit of measures.
\end{prop}

\begin{proof}

The key observation is, that if $\{\mu_N\}_{N\ge 1}$ is a consistent family of measures then for any fixed $t \ge 0$ , $\{\mu_NP_N(t)\}_{N\ge 1}$, i.e. the laws of $X^{(N)}(t)$ if $X^{(N)}(0)\overset{d}{=}\mu_N$, form a coherent sequence as well. This can be seen as follows,
\begin{align*}
\mu_{N+1}P_{N+1}(t) \Lambda_N^{N+1}=\mu_{N+1}\Lambda^{N+1}_NP_N(t)=\mu_N P_N(t),
\end{align*}
and moreover, if $\mu$ is the probability measure on $\Omega$ corresponding to $\{\mu_N\}_{N\ge 1}$ then we have,
\begin{align*}
\mu P_{\infty}(t)\Lambda_N^{\infty}=\mu\Lambda_N^{\infty}P_N(t)=\mu_NP_N(t).
\end{align*}

Thus, if the initial conditions converge as $N \to \infty$,
\begin{align*}
\alpha_i^+\left(x^{(N)};0\right)&\overset{d}{\longrightarrow} \alpha_i^+\left(0\right)\ , \forall i \ge 1,\\
\alpha_i^-\left(x^{(N)};0\right)&\overset{d}{\longrightarrow} \alpha_i^-\left(0\right)\ , \forall i \ge 1,\\
\gamma_1\left(x^{(N)};0\right)&\overset{d}{\longrightarrow} \gamma_1\left(0\right),\\
\delta\left(x^{(N)};0\right)&\overset{d}{\longrightarrow} \delta\left(0\right),
\end{align*}
where each $x^{(N)}$ is sampled according to the\textit{ coherent measures} $\mu_N$ and $\alpha_i^+\left(0\right),\alpha_i^-\left(0\right),\gamma_1\left(0\right),\delta\left(0\right)$ according to $\mu$ (we are abusing notation here, the parameter $0$ really corresponds to time and has nothing to do with $\omega$) then, for any \textit{fixed} $t\ge 0$ we have as $N \to \infty$,
\begin{align*}
\alpha_i^+\left(x^{(N)};t\right)&\overset{d}{\longrightarrow} \alpha_i^+\left(t\right)\ , \forall i \ge 1,\\
\alpha_i^-\left(x^{(N)};t\right)&\overset{d}{\longrightarrow} \alpha_i^-\left(t\right)\ , \forall i \ge 1,\\
\gamma_1\left(x^{(N)};t\right)&\overset{d}{\longrightarrow} \gamma_1\left(t\right),\\
\delta\left(x^{(N)};t\right)&\overset{d}{\longrightarrow} \delta\left(t\right),
\end{align*}
where, the $\left(\alpha^{\pm}_i(t),\gamma_1(t),\delta(t) \right)$ have the law of $\mu P_{\infty}(t)$ (or equivalently they are just $\mathsf{X}_{\infty}(t)$ written out in coordinates if $\mathsf{X}_{\infty}(0)\overset{d}{=}\mu$). This is exactly what we wanted to prove.
\end{proof}

The result above, although general might seem rather weak as a convergence statement but note however that since any point $\omega\in\Omega$ is given (by definition) by an extremal sequence of coherent probability measures Proposition \ref{propositionapproximation} completely characterizes the abstract semigroup $P_{\infty}(t)$ and thus also $\left(\mathsf{X}_{\infty}(t);t\ge 0\right)$. As we shall see in Subsection \ref{remarkDyson} below, much stronger convergence results can be obtained on a case by case basis.

\subsection{Dynamical systems on $\Omega$ coming from Dyson Brownian motions} \label{remarkDyson}

As already mentioned in the introduction, Dyson Brownian motions $(DBM)$ of different dimensions, given by the solution to the $SDEs$,
\begin{align*}
dX^N_i(t)=dW^N_i(t)+\sum_{j\ne i}^{}\frac{1}{X^N_i(t)-X^N_j(t)}dt,
\end{align*}
and with semigroups denoted by $P^N_{DBM}(t)$ are also consistent with the links $\Lambda_N^{N+1}$, see \cite{Warren},\cite{RamananShkolnikov}. We hence, again obtain a Feller-Markov process $\mathsf{X}^{DBM}_{\infty}$ on $\Omega$ that however has no invariant probability measure. We now describe the boundary process explicitly.

\begin{prop}
The process on $\Omega$ corresponding to Dyson Brownian motions:
\begin{align*}
\left(\mathsf{X}^{DBM}_{\infty}(t);t\ge 0\right)=\left(\alpha^{\pm}_i(t),\gamma_1(t),\gamma_2(t);t\ge 0\right)
\end{align*}
is given by, $ \forall  t \ge 0$:
\begin{align*}
\alpha_i^+(t)&=\alpha_i^+(0), \forall i \ge 1,\\
 \alpha_i^-(t)&=\alpha_i^-(0), \forall i \ge 1,\\
 \gamma_1(t)&=\gamma_1(0),\\
 \gamma_2(t)&=t+\gamma_2(0).
\end{align*}
Thus, it increases the Gaussian component linearly in time while it does nothing to the rest.
\end{prop}

\begin{proof}
We first show that, $\forall T>0$ we have,
\begin{align*}
\underset{1\le i \le N}{\max} \underset{0\le t \le T}{\sup}\left|\alpha_i^{+}\left(X^{(N)};t\right)-\alpha_i^{+}\left(X^{(N)};0\right)\right| \to 0 \textnormal{ almost surely as } N \to \infty.
\end{align*}
This can be seen as follows,
\begin{align*}
\frac{1}{N}\underset{1\le i \le N}{\max} \underset{0\le t \le T}{\sup}\left|\max \{X^{(N)}_{N+1-i}(t),0\}-\max \{X^{(N)}_{N+1-i}(0),0\}\right|&\le\frac{1}{N}\underset{1\le i \le N}{\max} \underset{0\le t \le T}{\sup}\left|X^{(N)}_{N+1-i}(t)-X^{(N)}_{N+1-i}(0)\right|\\
&\le\frac{1}{N}\underset{1\le i \le N}{\max} \underset{0\le t \le T}{\sup}\left|Y^{(N)}_{N+1-i}(t)\right|\\
&=\frac{1}{N}\max\{\underset{0\le t \le T}{\sup}Y^{(N)}_{N}(t),-\underset{0\le t \le T}{\inf} Y^{(N)}_1(t)\},
\end{align*}
where $Y^{(N)}$ is an $N$ particle $DBM$ starting from the origin. But by Theorem 3.7 of \cite{SupremumDBM} we have,
\begin{align*}
\frac{1}{\sqrt{N}}\underset{0\le t \le T}{\sup}Y^{(N)}_{N}(t) \to 2\sqrt {T} \textnormal{ almost surely as } N \to \infty,
\end{align*}
and similarly for $-\frac{1}{\sqrt{N}}\underset{0\le t \le T}{\inf}Y^{(N)}_{1}(t)$. The claim then follows and so since $T>0$ was arbitrary we obtain for $i \in \mathbb{N}$,
\begin{align*}
\alpha_i^+(t)=\alpha_i^+(0) \ , \forall t \ge 0.
\end{align*}
Analogously, for $i \in \mathbb{N}$,
\begin{align*}
\alpha_i^-(t)=\alpha_i^-(0) \ , \forall t \ge 0.
\end{align*}
We now have the following equation for $\gamma_1\left(X^{(N)};\cdot\right)$,
\begin{align*}
d\gamma_1\left(X^{(N)};t\right)=\frac{1}{N}\sum_{i=1}^{N}dW^N_i(t)=\frac{1}{\sqrt{N}}d\beta^N(t),
\end{align*}
where by Levy's characterization $\beta^N$ is a standard Brownian motion and thus as $N\to \infty$,
\begin{align*}
\gamma_1(t)=\gamma_1(0) \ , \forall t \ge 0.
\end{align*}
Finally, after an application of Ito's formula and some manipulations (see for example Step 2 of the proof of Theorem 1.1 in \cite{RamananShkolnikov} for the details) we obtain,
\begin{align*}
d\delta\left(X^{(N)};t\right)=\frac{1}{N^2}\left[N^2dt+2\sqrt{N^2\delta\left(X^{(N)};t\right)}d\tilde{\beta}^N(t)\right]=dt+\frac{1}{N}2\sqrt{\delta\left(X^{(N)};t\right)}d\tilde{\beta}^N(t),
\end{align*}
where $\tilde{\beta}^N$ is a standard Brownian motion. Thus, from Theorem 11.1.4 of \cite{StroockVaradhan} for example, we obtain,
\begin{align*}
\delta(t)=t+\delta(0),
\end{align*}
and so,
\begin{align*}
\gamma_2(t)=t+\gamma_2(0).
\end{align*}
\end{proof}

On the other hand, we could have considered a stationary or Ornstein-Uhlenbeck version of $DBM$. These are given by the solutions to the SDEs,
\begin{align*}
dX^N_i(t)=dW^N_i(t)+\left[-cX^N_i(t)+\sum_{j\ne i}^{}\frac{1}{X^N_i(t)-X^N_j(t)}\right]dt,
\end{align*}
and with semigroups denoted by $P^{c,N}_{OU}(t)$ they are consistent with the links, see \cite{RamananShkolnikov}. For each $N$, we have that $P^{c,N}_{OU}(t)$ has the $GUE_N$ ensemble with variance $\frac{1}{2c}$ as its unique invariant probability measure. Hence, the corresponding Markov process on $\Omega$ has as unique invariant measure a delta function concentrated at $\gamma_2(\omega)=\frac{1}{2c}$ with all the other coordinates $\gamma_1(\omega),\alpha^+_k(\omega),\alpha^{-}_k(\omega)$ being identically zero. Analogous considerations as for DBM, give the following differential equations for the $\alpha^{\pm}_i,\gamma_1$ and $\delta$,
\begin{align*}
\frac{d}{dt}\alpha^{\pm}_i(t)=-c\alpha_i^{\pm}(t) \ , \ \frac{d}{dt}\gamma_1(t)=-c\gamma_1(t) \ ,\ \frac{d}{dt}\delta(t)=(1-2c\delta(t)).
\end{align*}
Solving them, we obtain,
\begin{align*}
\alpha^{\pm}_i(t)=\alpha_i ^{\pm}(0)e^{-ct} \ , \gamma_1(t)=\gamma_1(0)e^{-ct} \ ,\ \delta(t)=\frac{1}{2c}\left(1-e^{-2ct}\right)+\delta(0)e^{-2ct},
\end{align*}
and so,
\begin{align*}
\gamma_2(t)=\frac{1}{2c}\left(1-e^{-2ct}\right)+\gamma_2(0)e^{-2ct}.
\end{align*}
Hence, as already observed above, we can easily see that the delta measure with $\gamma_2=\frac{1}{2c}$ and all other coordinates being 0 is the unique invariant measure and moreover the process converges exponentially fast to it.

\begin{rmk}
It is natural to try to apply the same scheme for the Hua-Pickrell diffusions. As expected, it can be seen at least formally that, in this case both the noise and the long range interactions will still be present in the limit $N\to \infty$ and we will be dealing with a truly infinite dimensional system of $SDEs$ (ISDE). Making rigorous sense of this is not straightforward, however there is some hope that one might be able to treat this with the general theory currently being developed  for such systems of ISDE by Osada and coworkers, see for example \cite{OsadaTanemura}.
\end{rmk}

\section{Dynamics on the path space of the graph of spectra}\label{sectionmultilevel}

\subsection{Multilevel interlacing dynamics}

The goal of this section is to construct a Markov process on the path space of the graph of spectra, such that the projection on level $N$ evolves according to $P_{HP}^{s,N}(t)$. The motivation behind this study is to provide a relation between the discrete dynamics introduced by Borodin and Olshanski in \cite{BorodinOlshanski} on the path space of the Gelfand-Tsetlin graph, that we will elaborate on later on, and the constructions of this paper.

Firstly, continuing with the graph analogy, if a "\textit{vertex}" at level $n$ of the \textit{graph of spectra} corresponds to a point $\left(x_1^{(n)},\cdots , x_n^{(n)}\right)$ in $W^n$, then a \textit{path} with $N$ steps is given by an interlacing array $\left(x_i^{(n)}, 1 \le i \le n \le N:x_i^{(n+1)}\le x_i^{(n)}\le x_{i+1}^{(n+1)}\right)$ or continuous Gelfand-Tsetlin pattern $\mathbb{GT}_c(N)$ with $N$ levels. 

For any $N\ge 1$, we can construct a Markov process on such \textit{paths} or equivalently a Markovian evolution taking values in the space of continuous Gelfand-Tsetlin patterns $\mathbb{GT}_c(N)$, as follows, for $1 \le i \le n \le N$, until a stopping time $\mathfrak{T}_{\mathbb{GT}_c(N)}$ (see below) given by,
\begin{align}\label{multilevelSDE}
dX^{(n)}_i(t)=\sqrt{2((X_i^{(n)})^2(t)+1)}d\beta^{(n)}_i(t)+\left[\left(2-2n-2\Re(s)\right)X_i^{(n)}(t)+2\Im(s)\right]dt+\frac{1}{2}dK_i^{(n),-}(t)-\frac{1}{2}dK_i^{(n),+}(t),
\end{align}
where $K_i^{(n),-}$ and $K_i^{(n),+}$ are the semimartingale local times of $X^{(n)}_i-X^{(n-1)}_{i-1}$ and $X^{(n)}_i-X^{(n-1)}_{i}$ at $0$ and $\beta_i^{(n)}$ for  $1 \le i \le n \le N$ are independent standard Brownian motions. The reader should note here, that the interaction is purely local and moreover that level $n$ given level $n-1$ is \textit{autonomous} consisting of $n$ independent $L_s^{(n)}$-diffusions that are kept apart by the \textit{random barriers} $\left(X_1^{(n-1)},\cdots,X_{n-1}^{(n-1)}\right)$.

There is a slight technical issue here, that corresponds to the fact that two paths at level $n$ (for some $n\le N$) might meet at the stopping time $\mathfrak{T}_{\mathbb{GT}_c(N)}$ given as,
\begin{align*}
\mathfrak{T}_{\mathbb{GT}_c(N)}=\inf\{t>0:\exists \ 1 \le i < j \le n\le N \textnormal {  s.t } X_i^{(n)}(t)=X_j^{(n)}(t)\},
\end{align*}
at which point we must stop the process. However, under some special initial conditions that we are about to define $\mathfrak{T}_{\mathbb{GT}_c(N)}=\infty$ almost surely and in particular the process in $\mathbb{GT}_c(N)$ has infinite lifetime.

Now, let $\nu_N(dx^{(N)})$ be a probability measure on $\mathring{W}^N$ and consider the following measure on $\mathbb{GT}_c(N)$ that we call \textit{central} or \textit{Gibbs},
\begin{align}\label{Gibbs}
\nu_N(dx^{(N)}) Uniform_{\mathbb{GT}_c(N)}^{x^{(N)}}(dx^{(1)},\cdots,dx^{(N-1)}),
\end{align}
where, 
\begin{align*}
Uniform_{\mathbb{GT}_c(N)}^{x^{(N)}}(dx^{(1)},\cdots,dx^{(N-1)})=\frac{\prod_{j=1}^{N-1}j!}{\Delta_N(x^{(N)})}\textbf{1}\left(x ^{(1)}\prec x ^{(2)}\prec \cdots \prec x ^{(N-1)}\prec x ^{(N)}\right)dx^{(1)}\cdots dx^{(N-1)},
\end{align*}
is the uniform distribution on $\mathbb{GT}_c(N)$ with fixed bottom row $x^{(N)}$.
Moreover, observe that:
\begin{align*}
Uniform_{\mathbb{GT}_c(N)}^{x^{(N)}}(dx^{(1)},\cdots,dx^{(N-1)})=\Lambda^N_{N-1}(x^{(N)},dx^{(N-1)})\cdots \Lambda_2^{3}(x^{(3)},dx^{(2)})\Lambda_1^{2}(x^{(2)},dx^{(1)}).
\end{align*}
Then, from Proposition 3.1 in \cite{InterlacingDiffusions} (or one could use the alternative approach of Sun \cite{Sun} combined with Theorem \ref{intertwiningtheorem} above) we obtain:

\begin{prop}
Assume that the system of SDEs with reflection (\ref{multilevelSDE}) is initialized according to a Gibbs measure, for $\nu_N(dx^{(N)})$ a probability measure supported on $\mathring{W}^N$:
\begin{align*}
\nu_N(dx^{(N)}) Uniform_{\mathbb{GT}_c(N)}^{x^{(N)}}(dx^{(1)},\cdots,dx^{(N-1)}).
\end{align*}
Then, the projection on the $n^{th}$ level evolves as a Markov process, with semigroup $P^{s,n}_{HP}(t)$, started according to $\left(\nu_N \Lambda^N_n\right)(dx^{(n)})$ i.e it evolves as,
\begin{align*}
dX^{(n)}_i(t)=\sqrt{2((X_i^{(n)})^2(t)+1)}dW^{(n)}_i(t)+\left[\left(2-2n-2\Re(s)\right)X_i^{(n)}(t)+2\Im(s)+\sum_{j\ne i}^{}\frac{2((X_i^{(n)}(t))^2+1)}{X_i^{(n)}(t)-X_j^{(n)}(t)}\right]dt,
\end{align*}
and in particular $\mathfrak{T}_{\mathbb{GT}_c(N)}=\infty$ almost surely. Moreover, the distribution of $\left(X^{(1)}(T),\cdots,X^{(N)}(T)\right)$ at fixed time $T\ge 0$ is still given by a Gibbs measure:
\begin{align*}
\left[\nu_NP^{s,N}_{HP}(T)\right](dx^{(N)}) Uniform_{\mathbb{GT}_c(N)}^{x^{(N)}}(dx^{(1)},\cdots,dx^{(N-1)}).
\end{align*}
\end{prop}

\subsection{Connection to dynamics for zw-measures on Gelfand-Tsetlin graph}
We now move on, to explain a relation between the dynamics on the path space of the Gelfand-Tsetlin graph constructed by Borodin and Olshanski and the dynamics on the path space of the graph of spectra considered here: under a spacial scaling limit they give rise to the same process in continuous Gelfand-Tsetlin patterns. The reader should note that our discussion below is informal and we shall prove no theorem, moreover the connection between the respective infinite dimensional processes on the boundaries remains mysterious.

We begin by explaining the dynamics of Borodin and Olshanski. First we will need to recall the bare minimum of definitions. A path of length $N$ in the Gelfand-Tsetlin graph is given by a Gelfand-Tsetlin pattern or scheme defined as follows. We will denote by $W^{n}(\mathbb{Z})=\{(x_1,\cdots,x_n)\in \mathbb{Z}^n: x_1 < \cdots < x_{n} \}$ ordered $n$-particle configurations and we will say that $y \in W^{n}(\mathbb{Z})$ and $x \in W^{n+1}(\mathbb{Z})$ interlace if $ x_1 \le y_1 < x_2 \le \cdots \le y_n< x_{n+1}$ and abusing notation we write $y\prec x$. Then, the space of Gelfand-Tsetlin patterns of depth (or height) $N$ is given by:
\begin{align}
\mathbb{GT}(N)=\big\{\left(x^1,\cdots,x^N\right): x^i \prec x^{i+1},\textnormal{ for } 1 \le i\le N-1 \big\}.
\end{align}

\paragraph{Borodin-Olshanski dynamics} The dynamics were introduced in Section 8 of \cite{BorodinOlshanski} and go as follows: each of the $n$ particles on level $n$ has two independent exponential clocks depending on its position $x \in \mathbb{Z}$ for jumping to the right by one with rate $\lambda_n(x)=\left(x-\left(u+n-1\right)\right)\left(x-\left(u'+n-1\right)\right)$ and to the left by one with rate $\mu_n(x)=\left(x+v\right)\left(x+v'\right)$. Here the parameters $u,u',v,v'\in \mathbb{C}$ satisfy certain constraints for the rates to be strictly positive and for the chain not to explode. 
In order for this Markov process to remain in $\mathbb{GT}(N)$ the particles interact through the so called push-block dynamics: There's a hierarchy for the particles, lower level ones can be thought of as heavier or more important. If the exponential clock for jumping to the right of the particle $X_k^n$ rings first, it attempts to jump to the right by one unit. It first looks at the $(n-1)^{th}$ level to check whether it is blocked, namely if $X_k^{n-1}=X_k^n$. In case it is, nothing happens, otherwise it moves by one to the right, possibly triggering some pushing moves. Namely if the interlacing is no longer preserved with the particle labelled $X_{k+1}^{n+1}$ then $X_{k+1}^{n+1}$ also moves (instantaneously) to the right by one. This pushing is propagated to higher levels.

\paragraph{Convergence of dynamics on path space}Intuitively the push-block dynamics are the discrete analogue of the local reflection interactions found in the $SDEs$ above, since particles interact only when the interlacing is about to be broken. The rigorous justification of this goes through the so called Skorokhod problem and usually requires substantial technical efforts and we will not pursue it here.

What we will do however is describe the motion of individual particles on each level under a scaling limit. We will consider the following discrete to continuous scaling limit $x \rightsquigarrow x/M$ and we send $M \to \infty$ for the dynamics on the Gelfand-Tsetlin graph. Note that we just scale space and not time. Then, we formally obtain, modulo the convergence of the discrete push-block dynamics to $SDEs$ with reflection, a process on the path space of the graph of spectra. Particles on level $n$ move according to a diffusion process $\left(G(t);t \ge 0\right)$ with generator:
\begin{align*}
x^2\frac{d^2}{dx^2}+\left(2-2n-\left(u+u'+v+v'\right)\right)x\frac{d}{dx}.
\end{align*}
This is actually a geometric Brownian motion and is given explicitly, in terms of a standard Brownian motion $\beta(t)$:
\begin{align*}
G(t)=G(0)\exp \left(\sqrt{2}\beta(t)+\left(1-2n-\left(u+u'+v+v'\right)\right)t\right).
\end{align*}
We now perform the same spacial, continuous to continuous in this case, scaling limit $x \rightsquigarrow x/M$ with $M\to \infty$ to the Hua-Pickrell dynamics introduced above. Particles on level $n$ will then follow a diffusion with generator:
\begin{align*}
x^2\frac{d^2}{dx^2}+\left(2-2n-2\Re(s)\right)x\frac{d}{dx}.
\end{align*}
The Markov process obtained then coincides with the one we get from the discrete to continuous limit with the identification $2\Re(s)=u+u'+v+v'$. In terms of $SDEs$ with reflection this multilevel process, see also Section 3.6 of \cite{InterlacingDiffusions}, is given by:
\begin{align*}
dX^{(n)}_i(t)=\sqrt{2}|X_i^{(n)}(t)|d\beta^{(n)}_i(t)+\left[\left(2-2n-2\Re(s)\right)X_i^{(n)}(t)\right]dt+\frac{1}{2}dK_i^{(n),-}(t)-\frac{1}{2}dK_i^{(n),+}(t).
\end{align*}

\subsection{Dynamics for multilevel CUE} \label{sectionCUE}
The purpose of this short subsection is to investigate how the results above transfer to the circle $\mathbb{T}$ under the Cayley transform. With $u=e^{i\theta}$ and $u=\frac{i-x}{i+x}$ and $x=i\frac{1-u}{1+u}$ we have,
\begin{align*}
x=\tan\left(\frac{\theta}{2}\right) \textnormal{ or } \theta=2\tan^{-1}(x).
\end{align*}
Then, applying Ito's formula, we get with $u_j^{(N)}(t)=e^{i\theta_j^{(N)}(t)}$ so that $\theta_i^{(N)}(t)=2\tan^{-1}\left(X_i^{(N)}(t)\right)$,
\begin{align}\label{circleSDE}
d\theta_i^{(N)}(t)
&=2\sqrt{2}\cos\left(\frac{\theta_i^{(N)}(t)}{2}\right)dW^{(N)}_i(t)+\bigg[\left(-4N-4\Re(s)\right)\sin\left(\frac{\theta_i^{(N)}(t)}{2}\right)\cos\left(\frac{\theta_i^{(N)}(t)}{2}\right)\nonumber\\&+4\Im(s)\cos^2\left(\frac{\theta_i^{(N)}(t)}{2}\right)+\sum_{j \ne i}^{}\frac{4}{\tan\left(\frac{\theta_i^{(N)}(t)}{2}\right)-\tan\left(\frac{\theta_j^{(N)}(t)}{2}\right)}\bigg]dt.
\end{align}
Thus, the process has generator acting on $C^2_c\left(W^N\left(-\pi,\pi\right)\right)$, twice continuously differentiable functions with compact support in $W^N$; this class of functions is sufficiently large to characterize the distribution of $\left(\theta_1^{(N)}(t),\cdots,\theta_N^{(N)}(t);t\ge 0\right)$ since neither $\pm\pi$ or $\partial W^N$ are ever reached (as these correspond to explosions to $\pm \infty$ and collisions for the $SDEs$ (\ref{HuaPickrellDiffusion})), given by the differential operator,
\begin{align*}
\mathfrak{L}_s^{(N)}=4\sum_{i=1}^{N}\cos^2\left(\frac{\theta_i}{2}\right)\partial^2_{\theta_i}+\sum_{i=1}^{N}\bigg[\left(-4N-4\Re(s)\right)\sin\left(\frac{\theta_i}{2}\right)\cos\left(\frac{\theta_i}{2}\right)+4\Im(s)\cos^2\left(\frac{\theta_i}{2}\right)+\\
\sum_{j \ne i}^{}\frac{4}{\tan\left(\frac{\theta_i}{2}\right)-\tan\left(\frac{\theta_j}{2}\right)}\bigg]\partial_{\theta_i}.
\end{align*}
This operator can in fact be written as an $h$-transform, as follows,
\begin{align*}
\mathfrak{L}_s^{(N)}=h^{-1}_N(\theta)\circ \sum_{i=1}^{N}L^{s,(N)}_{\theta_i}\circ h_N(\theta)-const_{N,s},
\end{align*}
where the one dimensional diffusion operators are given by,
\begin{align*}
L^{s,(N)}_{\theta_i}=4\cos^2\left(\frac{\theta_i}{2}\right)\frac{d^2}{d\theta_i^2}+\left[\left(-4N-4\Re(s)\right)\sin\left(\frac{\theta_i}{2}\right)\cos\left(\frac{\theta_i}{2}\right)+4\Im(s)\cos^2\left(\frac{\theta_i}{2}\right)\right]\frac{d}{d\theta_i},
\end{align*}
and the positive eigenfunction $h_N$,
\begin{align*}
h_N(\theta)=\prod_{1\le i < j \le N}^{}\left(\tan\left(\frac{\theta_j}{2}\right)-\tan\left(\frac{\theta_i}{2}\right)\right).
\end{align*}
The process (\ref{circleSDE}) above leaves $\mathfrak{C}_{*}\left(\mu_{HP}^{s,N}\right)$ invariant, in particular $CUE_N$ for $s=0$ .

\section{Matrix Hua-Pickrell Process} \label{sectionmatrixprocess}

In this section, we define a matrix process with its eigenvalues evolving according to (\ref*{HuaPickrellDiffusion}) and leaving the matrix Hua-Pickrell measure $\mathsf{M}_{HP}^{s,N}(dX)$, with $\Re(s)>-\frac{1}{2}$, defined in (\ref{MatrixHuaPickrellMeasure}) invariant. So, let $\left(\boldsymbol{B}^{(k)}_t;t\ge 0\right)$, for $k=1,2$, be two $N \times N$ matrices with entries independent standard Brownian motions. Moreover, define $\left(\boldsymbol{W}_t; t \ge 0\right)$ by $\boldsymbol{W}_t=\boldsymbol{B}^{(1)}_t+i\boldsymbol{B}_t^{(2)}$ and let $h:\mathbb{R}\to \mathbb{R}$, $g:\mathbb{R}\to \mathbb{R}$, $b:\mathbb{R}\to \mathbb{R}$ and $\alpha \in \mathbb{R}$. Consider the following stochastic process $\left(\boldsymbol{X}_t; t \ge 0\right)$, taking values in the space of $N \times N$ Hermitian matrices and verifying the matrix valued $SDE$,
\begin{align}
d\boldsymbol{X}_t=g(\boldsymbol{X}_t)d\boldsymbol{W}_th(\boldsymbol{X}_t)+h(\boldsymbol{X}_t)d\boldsymbol{W}_t^*g(\boldsymbol{X}_t)+\left(b(\boldsymbol{X}_t)+\alpha Tr\left(\boldsymbol{X}_t\right)\boldsymbol{I}\right)dt,
\end{align}
where $h(\boldsymbol{X}_t), g(\boldsymbol{X}_t),b(\boldsymbol{X}_t)$ are defined spectrally; more precisely for a diagonalization of a Hermitian matrix $H=U^*\Lambda U$ with $U \in \mathbb{U}(N)$ and $\Lambda=\textnormal{diag}(\lambda_1,\cdots,\lambda_N)$, $g(H)=U^*g(\Lambda) U$ where $g(\Lambda)=\textnormal{diag}\left(g(\lambda_1),\cdots,g(\lambda_N)\right)$. Define the function $G:\mathbb{R}\times \mathbb{R} \to \mathbb{R}$ given by,
\begin{align*}
G(x,y)=g^2(x)h^2(y)+g^2(y)h^2(x).
\end{align*}
Denote by $\left(\Lambda_t;t \ge 0\right)=\left(\lambda_1(t),\cdots,\lambda_N(t);t \ge 0\right)$, the projection on the eigenvalues of $\left(\boldsymbol{X}_t; t \ge 0\right)$. Then if $\boldsymbol{X}_0$ has distinct eigenvalues almost surely we obtain the following \textit{closed} (note there is no dependence on the eigenvectors) system of $SDEs$ for the eigenvalues where the $\{\beta_i\}_{i=1}^N$ are independent standard (real) Brownian motions,
\begin{align*}
d\lambda_i(t)=2h(\lambda_i(t))g(\lambda_i(t))d\beta_i(t)+\left(b(\lambda_i(t))+\alpha\sum_{k=1}^{N}\lambda_k(t)+2\sum_{k\ne i}^{}\frac{G\left(\lambda_i(t),\lambda_k(t)\right)}{\lambda_i(t)-\lambda_k(t)}\right)dt,
\end{align*}
up to the first collision time $\tau=\inf\{t\ge0:\exists \ i,j \textnormal{ such that } \lambda_i(t)=\lambda_j(t)\}$. This is essentially Theorem 4 of \cite{MatrixYamadaWatanabe}, with the only variation being that, we have the extra drift term $\alpha Tr\left(\boldsymbol{X}_t\right)\boldsymbol{I}$ which obviously gives the contribution $\alpha\sum_{k=1}^{N}\lambda_k(t)$ in the drift of the $SDEs$ for the eigenvalues. 

We now specialize to the case of interest and we take,
\begin{align*}
h(x)=\sqrt{\frac{1+x^2}{2}}, g(x)\equiv 1,b(x)=(1-N-2\Re(s))x+2\Im(s),\alpha = 1,
\end{align*}
so that $\left(\boldsymbol{X}_t; t \ge 0\right)$ satisfies,
\begin{align}\label{matrixhuapickrelldiffusion}
d\boldsymbol{X}_t=d\boldsymbol{W}_t\sqrt{\frac{I+\boldsymbol{X}_t^2}{2}}+\sqrt{\frac{I+\boldsymbol{X}_t^2}{2}}d\boldsymbol{W}_t^*+\left[(-N-2\Re(s))\boldsymbol{X}_t+2\Im(s)\boldsymbol{I}+ Tr\left(\boldsymbol{X}_t\right)\boldsymbol{I}\right]dt.
\end{align}
With some simple algebra, using the fact,
\begin{align*}
\sum_{k\ne i}^{}\frac{\lambda_k^2(t)}{\lambda_i(t)-\lambda_k(t)}=\sum_{k\ne i}^{}\frac{\lambda_i^2(t)}{\lambda_i(t)-\lambda_k(t)}-(N-2)\lambda_i(t)-\sum_{k=1}^{N}\lambda_k(t),
\end{align*}
we obtain the system of $SDEs$ (\ref{HuaPickrellDiffusion}),
\begin{align*}
d\lambda_i(t)=\sqrt{2(1+\lambda^2_i(t))}d\beta_i(t)+\left(2\Im(s)+\left(2-2N-2\Re(s)\right)\lambda_i(t)+\sum_{j\ne i}^{}\frac{2\left(1+\lambda^2_i(t)\right)}{\lambda_i(t)-\lambda_j(t)}\right)dt.
\end{align*}
Thus, the eigenvalues $\left(\Lambda_t;t \ge 0\right)$ of $\left(\boldsymbol{X}_t; t \ge 0\right)$ form a Hua-Pickrell diffusion. Moreover, since the system of $SDEs$ (\ref{HuaPickrellDiffusion}) has no collisions and does not explode we also get $\tau=\infty$ almost surely (this again can be seen in a couple of ways in analogy to Proposition \ref{FellerSemigroup} namely either using Theorem 2.2 of \cite{Graczyk}, which amounts to a classical argument due to McKean, or the fact that the process is a Doob $h$-transform of identical one dimensional diffusions killed when they intersect). We now prove the following properties for $\left(\boldsymbol{X}_t; t \ge 0\right)$.

\begin{lem}
The SDE (\ref{matrixhuapickrelldiffusion}) has a unique strong solution.
\end{lem}

\begin{proof}
First, note that we can write the matrix SDE (\ref{matrixhuapickrelldiffusion}) in vectorized form in terms of the real and imaginary parts of the entries. This gives a system of $N^2$ one-dimensional stochastic equations driven by $2N^2$ independent standard real Brownian motions, coming from $\boldsymbol{W}_t$. 

Now, observe that the drift term of the vectorized equation is clearly Lipschitz (it is just linear). Moreover, since $h(x)=\sqrt{\frac{1+x^2}{2}}$ is Lipschitz then, by Theorem 1.1 of \cite{MatrixLipschitz}, the matrix function $\boldsymbol{X} \mapsto \sqrt{\frac{I+\boldsymbol{X}^2}{2}}$ is also Lipschitz in any matrix norm and in particular in the Frobenius norm, namely the Euclidean norm of the vector of the entries. Hence, the diffusion term in this SDE system is Lipschitz as well. Thus, by Theorem 3.1 page 164 of \cite{IkedaWatanabe} for example, we obtain a unique strong solution to (\ref{matrixhuapickrelldiffusion}).
\end{proof}

\begin{prop}
For $\Re(s)>-\frac{1}{2}$, $\mathsf{M}_{HP}^{s,N}(dX)$ is invariant for $\left(\boldsymbol{X}_t; t \ge 0\right)$.
\end{prop}

\begin{proof}
Observe that this follows from the $\mathbb{U}(N)$-invariance of $\left(\boldsymbol{X}_t; t \ge 0\right)$ and the fact that $\left(\Lambda_t;t \ge 0\right)$, by Proposition \ref{invariant} has $\mu_{HP}^{s,N}$, with $\Re(s)>-\frac{1}{2}$, as its unique invariant measure. To see the $\mathbb{U}(N)$-invariance of $\left(\boldsymbol{X}_t; t \ge 0\right)$, define for $U \in \mathbb{U}(N)$ $\left(\boldsymbol{Y}_t; t \ge 0\right)=\left(U^*\boldsymbol{X}_tU; t \ge 0\right)$ and observe that $\left(\boldsymbol{Y}_t; t \ge 0\right)$ also satisfies (\ref{matrixhuapickrelldiffusion}),
\begin{align*}
d\boldsymbol{Y}_t=d\boldsymbol{\tilde{W}}_t\sqrt{\frac{I+\boldsymbol{Y}_t^2}{2}}+\sqrt{\frac{I+\boldsymbol{Y}_t^2}{2}}d\boldsymbol{\tilde{W}}_t^*+\left[(1-N-2\Re(s))\boldsymbol{Y}_t+2\Im(s)\boldsymbol{I}+ Tr\left(\boldsymbol{Y}_t\right)\boldsymbol{I}\right]dt,
\end{align*}
with $\left(\boldsymbol{\tilde{W}}_t;t \ge 0\right)=\left(U^*\boldsymbol{W}_tU;t\ge 0\right)\overset{law}{=}\left(\boldsymbol{W}_t;t\ge 0\right)$ by unitary invariance of Brownian motion, from which, if moreover $U^*\boldsymbol{X}_0U\overset{law}{=}\boldsymbol{X}_0$, the conclusion follows.
\end{proof}

We now give an alternative and rather neat proof for the fact that the semigroup $P^{s,N}_{HP}(t)$ has the Feller property, by appealing to known results. Since $\boldsymbol{X}_t$ solves an SDE with globally Lipschitz coefficients it is well known that it has the Feller property, see for example Theorem 19.9 of \cite{Schiling}. We denote by $\mathcal{S}^N(t)$ its semigroup. Note that the presence of the repulsive singular term does not allow us to apply this result directly to the eigenvalues. Moreover, observe that $f \mapsto f\circ \mathsf{eval}_N$ maps $C_0\left(W^N\right)$ to $C_0\left(H(N)\right)$.

\begin{prop}
The semigroup $P^{s,N}_{HP}(t)$, associated to $\mathsf{eval}_N(\boldsymbol{X}_t)$, has the Feller property.
\end{prop}
\begin{proof}
From the fact that the eigenvalue evolution is autonomous we obtain that $\forall f :W^N\to \mathbb{R}$ we have:
\begin{align*}
\mathcal{S}^N(t)\left(f \circ \mathsf{eval}_N\right)(H) \ \textnormal{only depends on } H \textnormal{ through } \mathsf{eval}_N(H).
\end{align*}
Namely, $\mathsf{eval}_N(\boldsymbol{X}_t)$ only depends on $H$ through $\mathsf{eval}_N(\boldsymbol{X}_0=H)$. Thus, if $x=\mathsf{eval}_N(H)$ we have:
\begin{align*}
\left[P^{s,N}_{HP}(t)f\right](x)=\left[\mathcal{S}^N(t)f\circ \mathsf{eval}_N\right](H)=\left[\mathcal{S}^N(t)f\circ \mathsf{eval}_N\right](U^*xU)\ , \forall U\in \mathbb{U}(N),
\end{align*}
where as before $\mathbb{U}(N)$ is the group of $N\times N$ unitary matrices. 
We proceed to check the Feller property. Since $x_n \to x \implies U^*x_n U \to U^* x U$ we get:
\begin{align*}
\left[P^{s,N}_{HP}(t)f\right](x_n)=\left[\mathcal{S}^N(t)f\circ \mathsf{eval}_N\right](U^*x_nU) \to \left[\mathcal{S}^N(t)f\circ \mathsf{eval}_N\right](U^*xU)=\left[P^{s,N}_{HP}(t)f\right](x).
\end{align*}
Moreover, since $x_n\to \infty \implies U^*x_nU \to \infty$ and $\left[\mathcal{S}^N(t)f\circ \mathsf{eval}_N\right]\in C_0\left(H(N)\right)$ we get:
\begin{align*}
\left[P^{s,N}_{HP}(t)f\right](x_n) \to 0 \ \textnormal{ as } \ x_n \to \infty.
\end{align*}
Finally we have continuity at $t=0$:
\begin{align*}
\underset{t \to 0}{\lim}\left[P^{s,N}_{HP}(t)f\right](x)=\underset{t \to 0}{\lim}\left[\mathcal{S}^N(t)f\circ \mathsf{eval}_N\right](U^*xU)=\left[f\circ \mathsf{eval}_N\right](U^*xU)=f(x).
\end{align*}
The proposition is fully proven.
\end{proof}

Before closing, we remark that under an application of the Cayley transform we obtain a process $\left(\boldsymbol{U}(t);t \ge 0\right)$ on the unitary group $\mathbb{U}(N)$ given by,
\begin{align*}
\boldsymbol{U}(t)=\mathfrak{C}(\boldsymbol{X})(t)=\frac{i-\boldsymbol{X}(t)}{i+\boldsymbol{X}(t)} \in \mathbb{U}(N),
\end{align*}
which has eigenvalues evolving according to $\left(e^{i\theta_1^{(N)}(t)},\cdots,e^{i\theta_N^{(N)}(t)};t \ge 0\right)$.

\begin{rmk}
In the special case $s=0$ note that $\left(\boldsymbol{U}(t);t \ge 0\right)$ is a $\mathbb{U}(N)$ valued process that the projection on its eigenvalues leaves $CUE_N$ invariant but itself is not unitary Brownian motion (and thus neither its spectrum follows circular Dyson Brownian motion abbreviated $cDBM$). In fact given that $cDBM$ can wrap around $\mathbb{T}$ such a multilevel construction of an interlacing process where the number of particles increases by one on each level does not seem possible (see Section 4 of \cite{Metcalfe} for example where a coupling is given for $n$ and $n$ particles of $cDBM$).
\end{rmk}

\section{Appendix}\label{appendix}
\subsection{Proof of intermediate intertwining relation}
In this appendix we give a self-contained proof for the intermediate intertwining relation (\ref{intermediateintertwining}). We essentially distil the arguments of \cite{InterlacingDiffusions} to the bare minimum required to give a proof of (\ref{intermediateintertwining}). 

\paragraph{Notation} We shall fix throughout this section the parameters $N\ge 1$ and $s\in \mathbb{C}$ and in order to ease notation we shall drop them from the superscripts and subscripts. For example, we will write $L$ for the generator of the one-dimensional Hua-Pickrell diffusion $L_s^{(N)}$, $p_t(x,y)$ for its transition density, instead of $p_t^{(N),s}(x,y)$, and so on. We also do the same for its dual $\hat{L}$. We will finally write $P(t)$ for the one-dimensional Feller semigroup the $L_s^{(N)}$-diffusion gives rise to (not to be confused with the $h$-transformed Karlin-McGregor semigroup $P_{HP}^{s,N}(t)$).

We also note that the results regarding smoothness and decay at $\pm \infty$ of the transition density $p_t(x,y)$ obtained in Lemma \ref{FellerSemigroup} also apply to the transition density $\hat{p}_t(x,y)$ of the dual $\hat{L}$-diffusion. We now arrive at a lemma which explains the importance of duality. Essentially, the commutation relation (\ref{SiegmundCommutation}) between one-dimensional differential operators, given in the proof below, is how one arrives at the definition of $\hat{L}$.

\begin{lem}[Siegmund duality] \label{DualityLemma}We have the following relation between the transition densities of the $L$-diffusion and its dual $\hat{L}$-diffusion:
\begin{align*}
-\partial_y\int_{-\infty}^{x}p_t(y,z)dz=\hat{p}_t(x,y).
\end{align*}
\end{lem}

\begin{proof}
We denote the left hand side of the equality by $q_t(x,y)$:
\begin{align*}
q_t(x,y)=-\partial_y\int_{-\infty}^{x}p_t(y,z)dz.
\end{align*}
We will now show that $q_t(x,y)$ solves the Kolmogorov forward equation for the $\hat{L}$-diffusion:
\begin{align*}
\partial_tq_t(x,y)&=\hat{L}_y^*q_t(x,y),  \ \ t>0, x,y \in \mathbb{R},\\
\lim_{t\to 0}q_t(x,y)&=\delta(x=y).
\end{align*}
Here, $\hat{L}_y^*$ denotes the (formal) adjoint of $\hat{L}$ with respect to Lebesgue measure acting in the variable $y$. Since, $-\infty $ and $+\infty$ are natural boundary points for the $\hat{L}$-diffusion, there are no further boundary conditions and the claim of the lemma follows.

For the time $t=0$ condition, we can easily see that formally:
\begin{align*}
\lim_{t \to 0}q_t(x,y)=-\partial_y \textbf{1}(y\le x)=\delta(x=y).
\end{align*}
The rigorous proof goes as follows. Let $f \in C_c^2\left(\mathbb{R}\right)$. Then, making use of the reversibility of $p_t(x,y)$ with respect to its speed measure $m$:
\begin{align*}
\frac{m(y)}{m(z)}p_t(y,z)=p_t(z,y)
\end{align*}
we calculate:
\begin{align*}
\int_{-\infty}^{\infty}q_t(x,y)f(y)dy&=\int_{-\infty}^{x}dz\int_{-\infty}^{\infty}-\partial_yp_t(y,z)f(y)dy=\int_{-\infty}^{x}dz\int_{-\infty}^{\infty}p_t(y,z)f'(y)dy\\
&=\int_{\infty}^{x}dz\int_{-\infty}^{\infty}\frac{m(z)}{m(y)}\frac{m(y)}{m(z)}p_t(y,z)f'(y)dy\\
&=\int_{\infty}^{x}m(z)dz\int_{-\infty}^{\infty}p_t(z,y)\frac{f'(y)}{m(y)}dy.
\end{align*}
Thus, we have:
\begin{align*}
\lim_{t\to 0}\int_{-\infty}^{\infty}q_t(x,y)f(y)dy&=\lim_{t\to 0}\int_{\infty}^{x}m(z)dz\int_{-\infty}^{\infty}p_t(z,y)\frac{f'(y)}{m(y)}dy\\
&=\int_{-\infty}^{x}m(z)\frac{f'(z)}{m(z)}dz=f(x).
\end{align*}
Finally, to show that $q_t(x,y)$ satisfies the differential equation we first note that by an elementary calculation:
\begin{align}\label{SiegmundCommutation}
\partial_y L_y=\hat{L}^*_y\partial_y.
\end{align}
Essentially this is how one arrives at the exact form of the dual diffusion $\hat{L}$. Then, using the fact that $p_t(x,y)$ solves the Kolmogorov backward differential equation for $L$:
\begin{align*}
\partial_tp_t(y,z)=L_{y}p_t(y,z), \ \ t>0,y,z\in \mathbb{R},
\end{align*}
 we have:
\begin{align*}
\partial_t q_t(x,y)=-\partial_y \int_{-\infty}^{x}\partial_tp_t(y,z)dz&=-\partial_y\int_{-\infty}^{x}L_yp_t(y,z)dz\\
&=-\hat{L}_y^*\partial_y\int_{-\infty}^{x}p_t(y,z)dz\\
&=\hat{L}_y^*q_t(x,y).
\end{align*}
\end{proof}

We now arrive at the following key definition of a block matrix determinant kernel ${q}_t^{N,N+1}((x,y),(x',y'))$. The reader is referred to \cite{InterlacingDiffusions} for motivation behind this and for a study of its remarkable probabilistic properties (that we will not need here).

\begin{defn}
Define the following block matrix determinant of size $(2N+1)\times (2N+1)$, with $t>0$, for $x,x'\in W^{N+1}$ and $y,y'\in W^N$ such that $y\prec x, y'\prec x'$:
\begin{align}
{q}_t^{N,N+1}((x,y),(x',y'))=\det\
 \begin{pmatrix}
{A}_t(x,x') & {B}_t(x,y')\\
  {C}_t(y,x') & {D}_t(y,y') 
 \end{pmatrix} \ 
\end{align}
where,
\begin{align*}
{A}_t(x,x')_{ij} &=\partial_{x'_j}P(t) \textbf{1}_{(-\infty,x_j']} (x_i)= p_t(x_i,x_j')  \ , \\
{B}_t(x,y')_{ij}&=\hat{m}(y'_j)(P(t) \textbf{1}_{(-\infty,y'_j]}(x_i) -\textbf{1}(j\ge i)) \ ,\\
{C}_t(y,x')_{ij}&=-\hat{m}^{-1}(y_i)\partial_{y_i}\partial_{x'_j}P(t) \textbf{1}_{(-\infty,x'_j]}(y_i) \ ,\\
{D}_t(y,y')_{ij}&=-\frac{\hat{m}(y'_j)}{\hat{m}(y_i)}\partial_{y_i} P(t) \textbf{1}_{(-\infty,y'_j]}(y_i)=\hat{p}_t(y_i,y_j').
\end{align*}
\end{defn}
Observe that, the second equality for the entries ${D}_t(y,y')_{ij}$ is really just the statement of Lemma \ref{DualityLemma}:
\begin{align*}
-\partial_{y_i}P(t)\textbf{1}_{(-\infty,y'_j]}(y_i)=\hat{p}_t(y_j',y_i)
\end{align*}
along with reversibility of the $\hat{L}$-diffusion with respect to its speed measure $\hat{m}$.

Moreover, observe that in the notation of (\ref{intermediateintertwining}):
\begin{align*}
\det\left({A}_t(x,x')_{ij}\right)_{i,j=1}^{N+1}&=\mathcal{P}^{(N+1)}_{s}(t)(x,x'),\\
\det\left({D}_t(y,y')_{ij}\right)_{i,j=1}^{N}&=\hat{\mathcal{P}}^{(N)}_{s}(t)(y,y').
\end{align*}

Finally, note that by the decay at $\pm \infty$ of the transition density and its derivatives obtained  in Lemma \ref{FellerSemigroup}, the determinant above, for any $x,y'$, is integrable in the variables $x'$ in the domain $y'\prec x'$ and $y$ in the domain $y\prec x$ (note that the latter is compact).
\begin{proof}[Proof of intertwining relation (\ref{intermediateintertwining})]
We will integrate the determinant ${q}_t^{N,N+1}((x,y),(x',y'))$ with respect to Lebesgue measure $dx'$ over $y'\prec x'$ and with respect to $\Lambda_{N,N+1}(x,dy)=\prod_{i=1}^{N}\hat{m}(y_i)\textbf{1}(y\prec x)dy$ (over $y\prec x$). Essentially, (\ref{intermediateintertwining}) follows immediately from computing this integral in two ways:
\begin{align*}
\int_{y'\prec x'}^{}dx'\int_{y\prec x}^{}dy \prod_{i=1}^{N}\hat{m}(y_i)q_t^{N,N+1}\left(\left(x,y\right),\left(x',y'\right)\right)=\int_{y\prec x}^{}\prod_{i=1}^{N}\hat{m}(y_i)dy\int_{y'\prec x'}^{}dx' q_t^{N,N+1}\left(\left(x,y\right),\left(x',y'\right)\right).
\end{align*}

We first perform the integration with respect to $dx'$. By multilinearity of the determinant we can bring the integrals inside since we note that for each $1\le j \le N+1$, the variable $x_j'$ only appears in a single column. Then, using the relations between the entries of the block matrix:
\begin{align*}
\int_{y_{j-1}'}^{y_j'}A_t\left(x,x'\right)_{ij}dx_j'&=\frac{1}{\hat{m}(y'_j)}B_t\left(x,y'\right)_{ij}-\frac{1}{\hat{m}(y_{j-1}')}B_t\left(x,y'\right)_{ij-1}+\textbf{1}\left(j=i\right), \textnormal{ for } 2\le j \le N, \\ 
\int_{y_{j-1}'}^{y_j'}C_t\left(y,x'\right)_{ij}dx_j'&=\frac{1}{\hat{m}(y'_j)}D_t\left(x,y'\right)_{ij}-\frac{1}{\hat{m}(y_{j-1}')}D_t\left(x,y'\right)_{ij-1}, \textnormal{ for } 2\le j \le N,\\
\int_{-\infty}^{y_1'}A_t\left(x,x'\right)_{i1}dx_1'&=\frac{1}{\hat{m}(y'_1)}B_t\left(x,y'\right)_{i1}+\textbf{1}\left(i=1\right), j=1,\\
 \int_{-\infty}^{y_1'}C_t\left(x,x'\right)_{i1}dx_1'&=\frac{1}{\hat{m}(y'_1)}D_t\left(x,y'\right)_{i1}, j=1,\\
\int_{y_N'}^{\infty}A_t\left(x,x'\right)_{iN+1}dx_{N+1}'&=-\frac{1}{\hat{m}(y'_N)}B_t\left(x,y'\right)_{iN}+\textbf{1}\left(i=N+1\right), j=N+1,\\ \int_{y_{N}'}^{\infty}C_t\left(x,x'\right)_{iN+1}dx_{N+1}'&=-\frac{1}{\hat{m}(y'_N)}D_t\left(x,y'\right)_{iN}, j=N+1,
\end{align*}
we easily get:
\begin{align*}
\int_{y'\prec x'}^{}dx'{q}_t^{N,N+1}((x,y),(x',y'))=\det\left({D}_t(y,y')_{ij}\right)_{i,j=1}^{N}=\hat{\mathcal{P}}^{(N)}_{s}(t)(y,y').
\end{align*}
Now, multiply both sides of the equality above by $\prod_{i=1}^{N}\hat{m}(y_i)dy$ and integrate over $y\prec x$. Then, the right hand side of the equality becomes:
\begin{align*}
\left[\Lambda_{N,N+1}\hat{\mathcal{P}}^{(N)}_{s}(t)\right](x,y').
\end{align*}
While, the left hand side using Fubini's theorem is:
\begin{align*}
\int_{y'\prec x'}^{}dx'\int_{y\prec x}^{}dy \prod_{i=1}^{N}\hat{m}(y_i)q_t^{N,N+1}\left(\left(x,y\right),\left(x',y'\right)\right).
\end{align*}

Now, the inner integral can be evaluated further: by multilinearity again, for each $1\le i \le N$, the variable $y_i$ only appears in a single row and using the relations: 
\begin{align*}
\int_{x_i}^{x_{i+1}}\hat{m}(y_i)C_{t}(y,x')_{ij}dy_i&=-A_t(x,x')_{i+1j}+A_t(x,x')_{ij},\\
\int_{x_i}^{x_{i+1}}\hat{m}(y_i)D_t(y,y')_{ij}dy_i&=-B_t(x,y')_{i+1j}+B_t(x,y')_{ij}+\hat{m}(y_j')\textbf{1}\left(j=i\right),
\end{align*}
we get:
\begin{align*}
\int_{y\prec x}^{}dy \prod_{i=1}^{N}\hat{m}(y_i)q_t^{N,N+1}\left(\left(x,y\right),\left(x',y'\right)\right)&=\det\left({A}_t(x,x')_{ij}\right)_{i,j=1}^{N+1}\prod_{i=1}^{N}\hat{m}(y_i')\\
&=\mathcal{P}^{(N+1)}_{s}(t)(x,x')\prod_{i=1}^{N}\hat{m}(y_i').
\end{align*}
Thus, the left hand side is equal to:
\begin{align*}
\left[\mathcal{P}^{(N+1)}_{s}(t)\Lambda_{N,N+1}\right](x,y')
\end{align*}
and so we obtain (\ref{intermediateintertwining}).
\end{proof}

\bigskip
\noindent
{\sc School of Mathematics, University of Bristol, U.K.}\newline
\href{mailto:T.Assiotis@bristol.ac.uk}{\small T.Assiotis@bristol.ac.uk}

\end{document}